\documentclass[lettersize,journal]{IEEEtran}
\usepackage{amsmath,amsfonts}
\usepackage{algorithmic}
\usepackage{algorithm}
\usepackage{multirow} 
\usepackage{array}
\usepackage{amsthm}
\usepackage[caption=false,font=normalsize,labelfont=sf,textfont=sf]{subfig}
\usepackage{textcomp}
\usepackage{stfloats}
\usepackage{url}
\usepackage{verbatim}
\usepackage{graphicx}
\usepackage{cite}
\usepackage{amssymb}
\usepackage[colorlinks,urlcolor=blue,citecolor=blue,linkcolor=blue]{hyperref}
\usepackage{etoolbox}
\makeatletter
\patchcmd{\@begintheorem}{\textit}{\textbf}{}{}
\makeatother

\DeclareMathOperator*\lowlim{\underline{lim}}
\DeclareMathOperator*\uplim{\overline{lim}}

\newcommand{\Real}{\mathbb{R}}

\newcommand{\Tra}{^{\sf T}} 

\newcommand{\Kron}{\otimes} 

\newcommand{\Sec}[1]{\hyperref[sec:#1]{Section~\ref*{sec:#1}}} 
\newcommand{\App}[1]{\hyperref[sec:#1]{Appendix~\ref*{sec:#1}}} 
\newcommand{\Supp}[1]{\hyperref[sec:#1]{Supplement~\ref*{sec:#1}}} 
\newcommand{\Eqn}[1]{\hyperref[eq:#1]{{\rm (\ref*{eq:#1})}}} 
\newcommand{\Part}[1]{\hyperref[part:#1]{(\ref*{part:#1})}} 
\newcommand{\Fig}[1]{\hyperref[fig:#1]{Figure~\ref*{fig:#1}}} 
\newcommand{\Tab}[1]{\hyperref[tab:#1]{Table~\ref*{tab:#1}}} 
\newcommand{\Thm}[1]{\hyperref[thm:#1]{Theorem~\ref*{thm:#1}}} 
\newcommand{\Lem}[1]{\hyperref[lem:#1]{Lemma~\ref*{lem:#1}}} 
\newcommand{\Prop}[1]{\hyperref[prop:#1]{Proposition~\ref*{prop:#1}}} 
\newcommand{\Cor}[1]{\hyperref[cor:#1]{Corollary~\ref*{cor:#1}}} 
\newcommand{\Def}[1]{\hyperref[def:#1]{Definition~\ref*{def:#1}}} 
\newcommand{\Alg}[1]{\hyperref[alg:#1]{Algorithm~\ref*{alg:#1}}} 
\newcommand{\Ex}[1]{\hyperref[ex:#1]{Example~\ref*{ex:#1}}} 
\newcommand{\As}[1]{\hyperref[as:#1]{Assumption~{\rm\ref*{as:#1}}}} 
\newcommand{\Reg}[1]{\hyperref[as:#1]{Condition~\ref*{reg:#1}}} 
\newcommand{\AlgLine}[2]{\hyperref[alg:#1]{line~\ref*{line:#2} of Algorithm~\ref*{alg:#1}}}
\newcommand{\AlgLines}[3]{\hyperref[alg:#1]{lines~\ref*{line:#2}--\ref*{line:#3} of Algorithm~\ref*{alg:#1}}}
\theoremstyle{definition}
\newtheorem{proposition}{Proposition}[section]
\newtheorem{theorem}{Theorem}[section]
\newtheorem{lemma}{Lemma}[section]
\newtheorem{definition}{Definition}
\newtheorem{remark}{Remark}

\begin{document}

\title{Anderson Accelerated Operator Splitting Methods for Convex-nonconvex Regularized Problems}

\author{Qiang Heng, Xiaoqian Liu, and Eric C. Chi

\thanks{Manuscript created December 2024. }
\thanks{Qiang Heng is with the School of Mathematics, Southeast University. Xiaoqian Liu is with the Department of Statistics, 
University of California, Riverside. Eric C. Chi is with the Department of Statistics, Rice University. The first two authors contributed equally to this article. }
}

\maketitle

\begin{abstract}
Convex-nonconvex (CNC) regularization is a novel paradigm that employs a nonconvex penalty function while maintaining the convexity of the entire objective function. It has been successfully applied to problems in signal processing, statistics, and machine learning. 
Despite its wide application, the computation of CNC regularized problems remains challenging and under-investigated.
To fill the gap, we study several operator splitting methods and their Anderson accelerated counterparts for solving least squares problems with CNC regularization. 
We establish the global convergence of the proposed algorithm to an optimal point and demonstrate its practical speed-ups in various applications. 
\end{abstract}


\section{Introduction}
A core task in signal processing, statistics, and machine learning is to solve a regularized least squares problem:
\begin{equation}
\label{eq:RLS}
    \underset{x\in \Real^p}{\min}\quad \frac{1}{2}\lVert y-Ax \rVert_2^2 + \lambda \psi(x),
\end{equation}
where $A$ is a matrix and $\psi$ is a regularizer that penalizes deviations away from a desired solution structure. The nonnegative tuning parameter $\lambda$ trades off the emphasis on model fit, quantified by the quadratic data fidelity term, and the model structure imposed by the regularizer $\psi$.

Convex regularizers such as the $\ell_1$-norm \cite{tibshirani1996regression}, total variation seminorm \cite{rudin1992nonlinear}, and the nuclear norm \cite{mazumder2010spectral} are the typical choice of a diverse array of classic signal processing and machine learning problems. Moreover, when $\psi$ is convex,  problem \Eqn{RLS} is convex and often admits scalable iterative algorithms with  convergence guarantees to global minima. Using a convex regularizer $\psi$, however, model parameters with large magnitudes are underestimated. This underestimation is referred to as ``shrinkage" bias in statistics. By contrast, nonconvex regularizers such as the smoothly clipped absolute deviation (SCAD) \cite{fan2001variable} and minimax concave penalty (MCP) \cite{zhang2010nearly} introduce less shrinkage bias while simultaneously producing sparser solutions. Employing a nonconvex regularizer, however, typically renders problem \Eqn{RLS} nonconvex. Unfortunately, iterative algorithms are typically no longer guaranteed to converge to global minima on nonconvex problems.

Convex-nonconvex (CNC) regularization, as a middle ground, is a way to leverage the strengths of convex and nonconvex regularizers while simultaneously mitigating their weaknesses\cite{blake1987visual, nikolova1998estimation,selesnick2014, selesnick2014convex,lanza2016convex, lanza2017nonconvex,selesnick2017sparse}. The idea is to design a nonconvex regularizer in a way that preserves the convexity of the objective function. 

In this paper, we focus on a popular formulation of CNC regularization proposed in \cite{selesnick2017sparse}.
\begin{equation}
\label{eq:cost}
\underset{x\in \Real^p}{\min}\quad h(x) = \frac{1}{2}\lVert y-Ax \rVert_2^2 + \lambda \psi_B(x),
\end{equation}
where
\begin{equation}
\label{eq:psiBdef}
\psi_B(x) = \rho(x) - \underset{v\in \Real^p}{\min}\;\left\{\rho(v) + \frac{1}{2}\lVert B(x-v)\rVert_2^2\right\}.
\end{equation}
Here, $\rho(x)$ is a convex function, and $B\in \Real^{n\times p}$ is a tuning parameter matrix specified by the user. The regularizer $\psi_B(x)$ is typically nonconvex, but the objective function $h(x)$ remains convex
when $B$ satisfies 
$A^\top A \succeq \lambda B^\top B$ \cite{selesnick2017sparse}.
For instance, 
\begin{equation}\label{eq:setB}
B = \sqrt{{\gamma}/{\lambda}}A,
\end{equation}
for some $\gamma \in [0, 1]$. Alternative strategies for setting $B$ in the context of image restoration have been explored in \cite{lanza2019sparsity}. For simplicity and consistency, throughout this paper, we assume that $B$ is specified as in \Eqn{setB}.

Various regularizers $\psi_B$ have been proposed in the literature. Choosing $\rho$ to be the $\ell_1$-norm leads to an extension to the MCP, the generalized minimax concave (GMC) penalty \cite{selesnick2017sparse}. Choosing $\rho$ to be the total-variation penalty \cite{rudin1992nonlinear}, a standard regularizer in signal and image denoising tasks, leads to estimators that better capture sharp transitions in signals more reliably
\cite{selesnick2017total, lanza2019sparsity}. Choosing $\rho$ to be the nuclear norm, a convex relaxation of the matrix rank, leads to improved methods for matrix completion \cite{lanza2019sparsity} and robust principal component analysis \cite{yin2019stable}. Choosing $\rho$ to be the $\ell_{2,1}$-norm \cite{liu2023convex, Chen2021}, leads to less biased grouped variable selection compared to the group Lasso penalty \cite{yuan2006model}. In all cases, $\psi_B$ achieves the same desired regularization effect as $\rho$ without its induced bias while also maintaining convexity of the objective function $h$.

Regarding its computation, problem \Eqn{cost} is typically solved via operator splitting methods. 
Setting $B$ as in \Eqn{setB}, problem \Eqn{cost} is equivalent to the following saddle point problem \cite{selesnick2017sparse}.
\begin{equation}\label{eq:minimax}
\min_{x\in \Real^p}\max_{v\in \Real^p}\quad H(x,v)
\end{equation}
where
\begin{equation}\label{eq:minimaxf}
H(x,v) = \frac{1}{2}\lVert y-Ax \rVert_2^2 + \lambda \rho(x) - \lambda \rho(v) - \frac{\gamma}{2} \lVert A(x-v) \rVert_2^2.
\end{equation}
We will see shortly that problem \eqref{eq:minimax} in turn is equivalent to an inclusion problem of the sum of two monotone operators. Therefore, operator splitting methods are natural strategies for solving \eqref{eq:minimax}. Interestingly, despite the availability of multiple splitting strategies,  \eqref{eq:minimax} has been mainly solved by forward-backward splitting (FBS)  \cite{combettes2011proximal} \cite{selesnick2017sparse,wang2018nonconvex,lanza2019sparsity,yin2019stable}. 
In this paper, we extensively explore potential operator-splitting-type algorithms for solving \Eqn{minimax}, including Douglas-Rachford splitting (DRS) \cite{douglas1956numerical, lions1979splitting, eckstein1992douglas}, forward-backward-forward splitting (FBFS) \cite{tseng2000modified}, and Davis-Yin splitting (DYS) \cite{davis2017three}. To further improve the efficiency, we investigate the Anderson acceleration (AA) \cite{Anderson65} of these operator-splitting methods. 

This paper makes the following contributions.
\begin{enumerate}
\item We demonstrate that FBFS can be a better choice than FBS for solving \eqref{eq:minimax} when the proximal operator is expensive to evaluate. We show that FBFS can take a larger step size than FBS and thus converges in fewer iterations.

\item  For CNC problems with two regularizers or an additional convex constraint, we are the first to show that the resulting three-operator monotone inclusion problem can be efficiently solved using DYS. 
  
\item We study Anderson accelerated operator splitting methods for solving CNC problems and show that the practical acceleration is significant over their un-accelerated counterparts. Furthermore, we theoretically ensure the global convergence of the accelerated algorithms through regularization and safeguarding. 
\end{enumerate}

\begin{remark} There has been several noteworthy extensions of problem \Eqn{cost}. The first is the linearly involved generalized Moreau enhanced (LiGME) 
model \cite{abe2020linearly,abe2019convexity,yata2022constrained}, which composes $\psi_B$ with an additional linear transformation matrix $L$. This enables the application of the construction technique of CNC to more general convex kernel functions. Another useful extension is the sharpening sparse regularizers (SSR) framework \cite{al2021sharpening}, where the infimal convolution in the regularizer differs from the LiGME model, which enables finer control of the shape of the regularizers. More recently, \cite{10251457} proposed the partially smoothed difference of convex regularizer (pSDC) model as a unified framework for convexity-preserving regularization. We remark that LiGME, SSR, and pSDC can all be considered as generalizations of the CNC model \Eqn{cost} and are interesting in their own right. The primary contribution of this paper is to introduce efficient new algorithms for solving \Eqn{cost}, since it is a fundamental model with many important applications, e.g., \cite{wang2018nonconvex,lanza2019sparsity,chatterjee2020review,yin2019stable,liu2023convex}. We note that the ideas and techniques in this work may be adapted to designing new algorithms for LiGME, SSR and pSDC, which we leave as a future work.
\end{remark}

\begin{remark} 
The LiGME model \cite{abe2020linearly} and the pSDC model \cite{10251457} also enable the incorporation of multiple regularizers. However, LiGME does so at the expense of lifting the problem to a higher-dimensional space, while pSDC generally requires iteratively solving two sub-problems at each difference-of-convex programming iteration. This paper presents a conceptually simple and computationally efficient alternative based on the Davis-Yin splitting.
\end{remark}

\begin{remark}
In \cite{10251457}, it was noted that AA can be applied to accelerate the inner loops of the difference-of-convex programming iterations. This work takes a different approach; AA is directly applied to the proximal-splitting iterations. Moreover, we use regularized and safeguarded AA to ensure the global convergence of the accelerated algorithms. We note that in general, naively applying AA may result in a divergent algorithm, even if the original fixed-point iteration is globally convergent (see \App{patho} for an example of this phenomenon). Additionally, regularization and safeguarding may also stabilize the algorithm and facilitate convergence in fewer iterations. We illustrate this improved convergence in  \Sec{sgl}. 
\end{remark}

{\bf Organization.} The rest of this article is organized as follows. In \Sec{splitting}, we review concepts from convex analysis and monotone operator theory used in this paper. We then formulate the CNC problem \Eqn{cost} as a monotone-inclusion problem and review four operator splitting methods for solving the inclusion problem. In \Sec{aa}, we introduce AA and describe a unified acceleration framework for the four operator splitting methods.  In \Sec{numexp}, we provide several numerical examples that demonstrate the practical speed-ups achieved by the proposed algorithms. In \Sec{discuss}, we close with a discussion. Proofs are in the appendix.

\section{Operator Splitting Methods for CNC}\label{sec:splitting}
\subsection{Concepts from Convex Analysis}\label{sec:monotone}
We start with reviewing concepts from convex analysis and monotone operator theory used later in this article. Let $\mathcal{P}: \mathbb{R}^p \to 2^{\mathbb{R}^p}$ denote a set-valued operator that assigns to each $z \in \mathbb{R}^p$ a (possibly empty) subset of $\mathbb{R}^p$. The domain of $\mathcal{P}$ is $\text{dom}\; \mathcal{P} = \{z \in \mathbb{R}^p | \mathcal{P}z \neq \varnothing\}$. When $\mathcal{P}$ is single-valued and has full domain ($\text{dom} \; \mathcal{P}=\Real^p$), we write $\mathcal{P}: \Real^p \rightarrow \Real^p$.

The \textit{graph} of a set-valued operator $\mathcal{P}$ is 
$$
\text{gra} \;\mathcal{P} = \{(z,u)\in \Real^p\times \Real^p| u \in \mathcal{P}z \}. 
$$
A set-valued operator $\mathcal{P}: \mathbb{R}^p \to 2^{\mathbb{R}^p}$ is said to be \textit{monotone} if 
$$
(z-z')^\top (u-u') \ge 0, \;\forall (z,u),(z',u') \in \text{gra} \;\mathcal{P}.
$$
A monotone operator $\mathcal{P} : \Real^p \rightarrow 2^{\Real^p}$ is \textit{maximal monotone} if its graph is not a proper subset of the graph of another monotone operator.

The \textit{inverse} of $\mathcal{P}$, denoted by $\mathcal{P}^{-1}: \mathbb{R}^p \to 2^{\mathbb{R}^p}$, is the set-valued mapping
$$
\mathcal{P}^{-1} z =\{u|z\in \mathcal{P}u\}.
$$
The inverse operator of a maximal monotone operator is also maximal monotone. The set of zeros for the operator $\mathcal{P}$ is denoted by $\text{zer}\; \mathcal{P} = \mathcal{P}^{-1}0$. 

The \textit{resolvent} of $\mathcal{P} : \Real^p \rightarrow 2^{\Real^p}$ is $\mathcal{J}_{\mathcal{P}} = (\mathcal{I}+\mathcal{P})^{-1}$ where $\mathcal{I}$ is the identity mapping. A fundamental result in convex analysis is that the resolvent of a maximal monotone operator is single-valued and has full domain (see \cite[Corollary 23.10]{bauschke2011convex}). 

$\mathcal{P} : \Real^p \rightarrow \Real^p$ is \textit{firmly nonexpansive} if
$$
(z-z')^\top (\mathcal{P}z - \mathcal{P}z') \ge \lVert \mathcal{P}z - \mathcal{P}z'\rVert_2^2 \quad \forall\;z,z' \in \Real^p.
$$
Additonally, $\mathcal{P}$ is \textit{$\beta$-cocoercive} if $\beta \mathcal{P}$ is firmly nonexpansive. 

$\mathcal{P} : \Real^p \rightarrow \Real^p$ is \textit{$L$-Lipschitz} if 
$$
\lVert \mathcal{P}z - \mathcal{P}z'\lVert_2\le L \lVert z-z' \rVert_2, \quad
\forall\;z,z' \in \Real^p.
$$
In particular, if $\mathcal{P}$ is $\beta$-cocoercive, then $\mathcal{P}$ is $1/\beta$-Lipschitz. The converse is not true in general \cite{o2020equivalence}.

\subsection{CNC as Monotone Inclusion}\label{sec:split}
The optimality conditions of \Eqn{minimax} are given by the inclusions
\begin{align}
\begin{split}\label{eq:subgradient}
0 \in \partial_x H(x,v) &= A^\top (Ax-y) - \gamma A^\top A(x-v) + \lambda \partial \rho(x),\\
0 \in \partial_v H(x,v) &= \gamma A^\top A(x-v) - \lambda \partial \rho(v),
\end{split}
\end{align}
where $\partial$ denotes the subdifferential operator. The inclusion condition \eqref{eq:subgradient} is  equivalent to the monotone inclusion problem
\begin{equation}\label{eq:operators}
\text{find } (x,v)\in \Real^{2p} \text{ s.t. } 0\in \mathcal{P}(x,v)+\mathcal{Q}(x,v)
\end{equation} where
\begin{align*}
\begin{split}
\mathcal{P}(x,v) &= \left(\begin{bmatrix}
1-\gamma & \gamma \\
-\gamma  & \gamma
\end{bmatrix}\Kron A^\top A\right)\begin{bmatrix}
x\\ v
\end{bmatrix} -\begin{bmatrix}
A^\top y\\ 0
\end{bmatrix},\\
\mathcal{Q}(x,v) &= \begin{bmatrix}
\lambda \partial \rho(x)
\\
\lambda \partial \rho(v)
\end{bmatrix}.
\end{split}
\end{align*}
\begin{proposition}\label{prop:satisfy}
\textit{Consider 
$\mathcal{P}$ and $\mathcal{Q}$ in \Eqn{operators}. Then
\begin{enumerate}
\item $\mathcal{P}$, $\mathcal{Q}$, and $\mathcal{P}+\mathcal{Q}$ are maximal monotone.  
\item $\mathcal{P}$ is $\beta$-cocoercive with $\beta = \min\{1,\frac{1-\gamma}{\gamma}\}/\lVert  A\rVert_2^2$ and $L$-Lipschitz with $L=\left\| \begin{bmatrix} 1-\gamma & \gamma \\ -\gamma & \gamma\end{bmatrix} \right\|_2\lVert  A\rVert_2^2$. 
\end{enumerate}}
\end{proposition}

In some applications we may be interested in using a mix of two convex regularizers $\lambda_1\rho_1 +\lambda_2\rho_2$.
In this case, the CNC problem \eqref{eq:cost} is equivalent to the saddle point problem $\min_x\max_v H(x,v)$, where
\begin{align}
\label{eq:twominimax}
\begin{split} H(x,v)  :=  \frac{1}{2}\lVert y-Ax \rVert_2^2 + \lambda_1 \rho_1(x) + \lambda_2 \rho_2(x)\\- \lambda_1 \rho_1(v) - \lambda_2 \rho_2(v) -\frac{\gamma}{2} \lVert A(x-v) \rVert_2^2.
\end{split}
\end{align}
It is straightforward to verify that the optimality condition for \Eqn{twominimax} is equivalent to a monotone inclusion problem with three operators.
\begin{equation}\label{eq:operators3}
\text{find } (x,v)\in \Real^{2p} \text{ s.t. } 0\in \mathcal{P}(x,v)+\mathcal{Q}(x,v)+\mathcal{R}(x,v)
\end{equation} where
\begin{align*}
\begin{split}\label{eq:PQR}
\mathcal{P}(x,v) & = \left(\begin{bmatrix}
1-\gamma & \gamma \\
-\gamma  & \gamma
\end{bmatrix}\Kron A^\top A\right)\begin{bmatrix}
x\\ v
\end{bmatrix} -\begin{bmatrix}
A^\top y\\ 0
\end{bmatrix},
\\
\mathcal{Q}(x,v) & = \begin{bmatrix}
\lambda_1 \partial \rho_1(x)
\\
\lambda_1 \partial \rho_1(v)
\end{bmatrix},\quad \mathcal{R}(x,v) = \begin{bmatrix}
\lambda_2 \partial \rho_2(x)
\\
\lambda_2 \partial \rho_2(v)
\end{bmatrix}.
\end{split}
\end{align*}

Another interesting setting is that one may want to impose a convex set constraint $x\in C\subset \Real^p$ in  problem \eqref{eq:cost}, which leads to
\begin{equation}
\label{eq:cost2}
\underset{x\in C}{\min}\quad h(x) = \frac{1}{2}\lVert y-Ax \rVert_2^2 + \lambda \psi_B(x).
\end{equation}
Problem \eqref{eq:cost2} is also equivalent to the three-operator monotone inclusion problem \Eqn{operators3}, but the operators $\mathcal{Q}$ and $\mathcal{R}$ are now
\begin{align*}
\mathcal{Q}(x,v) & = \begin{bmatrix}
\lambda \partial \rho(x)
\\
\lambda \partial \rho(v)
\end{bmatrix},\quad \mathcal{R}(x,v) = \begin{bmatrix}
 \partial \iota_C(x)
\\
0
\end{bmatrix},
\end{align*}
where $\iota_C$ is the indicator function for the convex set $C$.
\subsection{Operator Splitting Methods}
We review four operator-splitting methods and their fixed point iterations (FPI) for solving \Eqn{operators} and \Eqn{operators3}. Let $\mathcal{J}_{\mu \mathcal{T}}$ denote the resolvent of operator $\mu\mathcal{T}$, i.e., $\mathcal{J}_{\mu\mathcal{T}}z = (\mathcal{I}+\mu\mathcal{T})^{-1}z$. For clearer presentation, we write \Eqn{operators} and \eqref{eq:operators3} as 
\begin{equation}
\label{eq:twoOS}
\text{find $z\in \Real^{2p}$ such that }0\in \mathcal{P}z + \mathcal{Q}z,
\end{equation}
and 
\begin{equation}\label{eq:threeOS}
\text{find $z\in \Real^{2p}$ such that }0\in \mathcal{P}z + \mathcal{Q}z + \mathcal{R}z.
\end{equation}


\noindent \textbf{Douglas-Rachford Splitting} 
When $\mathcal{J}_{\mu\mathcal{P}}$ and $\mathcal{J}_{\mu\mathcal{Q}}$ are both easy to compute, DRS is a standard algorithm for solving \Eqn{twoOS}. The iterates are given by
\begin{align}\label{eq:driteration}
\begin{split}
z^{k+\frac{1}{3}} &= \mathcal{J}_{\mu\mathcal{P}}z^k,\\
z^{k+\frac{2}{3}} &= \mathcal{J}_{\mu\mathcal{Q}}(2z^{k+\frac{1}{3}}-z^k),\\
z^{k+1} &= z^k - z^{k+\frac{1}{3}} + z^{k+\frac{2}{3}},
\end{split}
\end{align}
and global convergence is guaranteed when $\mu > 0$. The FPI of DRS can be written as $F_{\text{DRS}} = \mathcal{I} - \mathcal{J}_{\mu\mathcal{P}}+ \mathcal{J}_{\mu\mathcal{Q}} \circ (2\mathcal{J}_{\mu\mathcal{P}}-\mathcal{I})$. 

\noindent \textbf{Forward-backward Splitting} When $\mathcal{P}$ is single-valued and $\beta$-cocoercive, another common splitting method to solve \Eqn{twoOS} is  FBS. The iterates are given by
\begin{align}\label{eq:fbiteration}
\begin{split}
z^{k+\frac{1}{2}} &= (\mathcal{I}-\mu \mathcal{P}) z^{k},\\
z^{k+1} &= \mathcal{J}_{\mu\mathcal{Q}} z^{k+\frac{1}{2}},
\end{split}
\end{align} 
and global convergence is guaranteed when $\mu \in (0,2\beta)$. The FPI of FBS can be written as $F_{\text{FBS}}= \mathcal{J}_{\mu\mathcal{Q}}\circ(\mathcal{I}-\mu \mathcal{P})$.

\noindent \textbf{Forward-backward-forward Splitting} FBFS, also called Tseng's algorithm \cite{tseng2000modified}, is similar to FBS but differs in two key aspects. First, it is applicable as long as $\mathcal{P}$ is $L$-Lipschitz, which is a weaker condition than cocoercivity. Second, it requires two forward operations per iteration.  The iterates of FBFS for solving \Eqn{twoOS} are given by
\begin{align}\label{eq:fbfiteration}
\begin{split}
z^{k+\frac{1}{3}} &= (\mathcal{I}-\mu \mathcal{P}) z^{k},\\
z^{k+\frac{2}{3}} &= \mathcal{J}_{\mu\mathcal{Q}} z^{k+\frac{1}{3}},\\
z^{k+1} &= z^k - z^{k+\frac{1}{3}} + (\mathcal{I}-\mu \mathcal{P})z^{k+\frac{2}{3}},
\end{split}
\end{align}
and global convergence is guaranteed when $\mu\in(0,1/L)$. 
The FPI of FBFS can be written as $F_{\text{FBFS}} = \mu \mathcal{P} + (\mathcal{I}-\mu \mathcal{P})\circ \mathcal{J}_{\mu\mathcal{Q}}\circ (\mathcal{I}-\mu \mathcal{P})$. 

\noindent \textbf{Davis-Yin Splitting} DYS can be applied to solve \Eqn{threeOS} when $\mathcal{P}$ is single-valued and $\beta$-cocoercive. The iterates are given by 
\begin{align}\label{eq:dyiteration}
\begin{split}
z^{k+\frac{1}{3}} &= \mathcal{J}_{\mu\mathcal{R}} z^{k},\\ 
z^{k+\frac{2}{3}} &= \mathcal{J}_{\mu\mathcal{Q}} (2z^{k+\frac{1}{3}}-z^k-\mu \mathcal{P}z^{k+\frac{1}{3}}),\\
z^{k+1} &= z^k - z^{k+\frac{1}{3}} + z^{k+\frac{2}{3}},
\end{split}
\end{align} 
and global convergence is guaranteed when $\mu \in (0,2\beta)$. 
The FPI of DYS can be written as $F_{\text{DYS}} = \mathcal{I} - \mathcal{J}_{\mu \mathcal{R}}+\mathcal{J}_{\mu \mathcal{Q}}\circ(2\mathcal{J}_{\mu \mathcal{R}} - \mathcal{I} - \mu \mathcal{P}\circ \mathcal{J}_{\mu \mathcal{R}})$. 

\section{Anderson Acceleration}\label{sec:aa}
Anderson Acceleration (AA) is a classic technique for expediting the convergence of fixed point iterations (FPI) \cite{Anderson65}. In the last decade, significant progress has been made in understanding AA from multiple perspectives. These include connections to the generalized minimal residual algorithm (GMRES) \cite{saad1986gmres, walker2011anderson, potra2013characterization}, connections to quasi-Newton methods \cite{fang2009two}, convergence analysis \cite{toth2015convergence,chen2019convergence, de2022linear}, and insights into acceleration behavior \cite{evans2020proof,rebholz2023effect}. 

Much of the recent interest in AA comes from solving optimization problems. Scieur et al.\@ \cite{scieur2016regularized} first considered the application of AA to general unconstrained optimization problems. Subsequently, AA has been used to accelerate EM algorithms \cite{henderson2019damped}, proximal gradient methods \cite{mai2020anderson}, coordinate descent methods \cite{bertrand2021anderson}, the Alternating Direction Method of Multipliers (ADMM) \cite{zhang2019accelerating, wang2021asymptotic},  gradient descent-ascent  for solving smooth minimax problems  \cite{he2022gda}, stochastic gradient descent for deep learning \cite{wei2021stochastic}, and momentum and primal-dual algorithms including Nesterov's method and primal-dual hybrid gradient (PDHG)\cite{bollapragada2023nonlinear}. The application of AA in different optimization problems is often accompanied by custom modifications, such as damping and preconditioning. 

Zhang et al.\@ \cite{zhang2020globally} introduced the first globally convergent version of AA, known as type-I AA, which requires tuning  hyperparameters to achieve good performance \cite{fu2020anderson}. To address this, \cite{fu2020anderson} proposed a globally convergent type-II AA algorithm for Douglas-Rachford splitting that delivers robust performance while mostly tuning-free.

In this paper, we generalize the Anderson accelerated Douglas-Rachford splitting algorithm proposed in \cite{fu2020anderson} to provide a unified framework of type-II AA for DRS, FBS, FBFS, and DYS. Expanding the existing framework from \cite{fu2020anderson} to include additional splitting methods is crucial for our goal of expediting the computation of CNC problems. As we will see later, in most cases, these alternative splitting methods offer superior computational and memory efficiency compared to the standard Douglas-Rachford splitting.

\subsection{Standard Form of AA}
We now introduce the standard (type-II) AA algorithm as described in \cite{Anderson65,walker2011anderson}. Consider a fixed-point mapping $F: \Real^{p}\rightarrow  \Real^{p}$.  The residual function of the FPI is defined as $G(z) = z - F(z)$. The memory size is a positive integer $M$, and $M^k = \min\{k,M\}$ denotes the memory size after $k$ iterations. At iteration $k$,  classic AA solves the following constrained least squares problem to compute a set of weights $\alpha_j^k, j = 0,1,\dots, M^k$:
\begin{align}\label{eq:aa2} 
\min_{\alpha^k \in \Real^{M^k+1}} \quad \left\lVert \sum_{j=0}^{M^k} \alpha_j^k G(z^{k-M^k+j}) \right\rVert_2^2,~~ \text{s.t.} \quad \sum_{j=0}^{M^k} \alpha_j^k = 1. 
\end{align}
The next iterate $z^{k+1}$ is obtained by combining the previous iterates according to the calculated weights, i.e. $z^{k+1} = \sum_{j=0}^{M^k}\alpha_j^kF(z^{k-M^k+j})$. After computing $z^{k+1}$, AA updates its memory to track  the $M^k$ most recent iterates, i.e., $\left(F\left(z^{k+1}\right), \dots, F\left(z^{k+1-M^{k}}\right)\right)$. 

\subsection{Regularization and Safeguarding}

The acceleration framework described in \Eqn{aa2} can be applied to almost any FPI, and empirical convergence is commonly observed. This raises a natural question: is \Eqn{aa2} globally convergent whenever the original FPI is globally convergent? Mai and Johansson \cite{mai2020anderson} demonstrated that this is not the case in general with a counterexample where AA iterates settle into in a periodic orbit. They restored global convergence by adding safeguard steps in the context of proximal gradient descent. Their safeguarding technique and global convergence results, however, rely on the strong convexity of the objective function, which is a strong assumption.

Fu et al.\@ \cite{fu2020anderson} introduced a globally convergent AA algorithm for DRS 
\Eqn{driteration}, using an adaptive regularization technique and incorporating safeguarding steps based on the norm of $G(z)$. Unlike the algorithm proposed in \cite{mai2020anderson}, the AA algorithm presented in \cite{fu2020anderson} only requires the original assumptions of DRS made to ensure global convergence. Notably, these assumptions do not include strong convexity. In fact, it does not restrict the optimization problem to be a minimization problem. Thus, the algorithm in \cite{fu2020anderson} can be applied to solve our problems of interest. However, DRS can be slow in our context due to the expensive matrix inversion that occurs in the evaluation of $\mathcal{J}_{\mu \mathcal{P}}$. To facilitate efficient solutions to CNC problems, we generalize the framework in \cite{fu2020anderson} to include several other splitting schemes, namely FBS \Eqn{fbiteration}, FBFS \Eqn{fbfiteration}, and DYS \Eqn{dyiteration}. 

Let $g^k = G(z^k)$, $y^k=g^{k+1}-g^k$, $s^k = z^{k+1}-z^k$. Denote the history of $y^k$ with $Y_k = \left[y^{k-M^k},\dots,y^{k-1}\right]$ and likewise the history of $s^k$ with $S_k$ = $\left[s^{k-M^k},\dots,s^{k-1}\right]$. Following a reparameterization in \cite{fang2009two}, problem \Eqn{aa2} is equivalent to the following unconstrained least squares problem
\begin{equation}\label{eq:aa2uc}
\min_{\zeta^k \in \Real^{M^k}}\quad \lVert g^k - Y_k\zeta^k\rVert_2,
\end{equation}
where
\begin{align}\label{eq:alphakformula}
\begin{split}
\alpha_0^k = \zeta_0^k,\quad  \alpha_i^k = \zeta_i^k-\zeta_{i-1}^k,\quad \alpha_{M^k}^k = 1-  \zeta_{M^k-1}^k,
\end{split}
\end{align}
and $i=1,2,\dots,M^k-1$.

The key to establishing the converge guarantees in \cite{fu2020anderson} are i) adaptive regularization and ii) safeguarding. Our generalization also relies on these techniques, so we review them now. Fu et al.\@ \cite{fu2020anderson} modified the classic AA and solved a regularized least squares problem in lieu of \Eqn{aa2}.
\begin{equation}\label{eq:aa2ucreg}
\min_{\zeta^k \in \Real^{M^k}}\quad\left\lVert g^k - Y_k\zeta^k\right\rVert_2^2+\eta\left(\lVert S_k\rVert_{\text{F}}^2+\lVert Y_k\rVert_{\text{F}}^2\right)\left\lVert \zeta^k\right\rVert_2^2,
\end{equation}
where $\eta>0$ is a small positive constant. The regularization technique \Eqn{aa2ucreg} is adaptive. As $\lVert S_k\rVert_{\text{F}}^2$ and $\lVert Y_k\rVert_{\text{F}}^2$ tend towards zero, the regularization vanishes. This is in contrast to the Tikhonov regularization approach employed in \cite{scieur2016regularized}.  The least $\ell_2$-norm solution to \Eqn{aa2ucreg} is
\begin{equation}\label{eq:gammakformula}
\zeta^k = \left(Y_k^\top Y_k + \eta(\lVert S_k\rVert_{\text{F}}^2+\lVert Y_k\rVert_{\text{F}}^2)I\right)^{\dagger} Y_k^\top g^k.
\end{equation}
Moreover, the relationship between $\zeta^k$ and $\alpha^k$ provides a convenient update rule for $z^{k}$.

\Alg{a2os} presents pseudo code for our safeguarded AA algorithm. In the algorithm, we use $Z_k$ to denote the matrix of past iterates $[F(z^{k-M^k}),\dots,F(z^{k})]$. Note that while $Y_k$ and $S_k$ have $M^k$ columns, $Z_k$ has $M^k+1$ columns.

\subsection{Main Algorithm}
\begin{algorithm}[htbp]
\caption{Anderson Accelerated Operator Splitting (A2OS)}\label{alg:a2os}
\begin{algorithmic}[1]
\REQUIRE Splitting scheme $\text{OS}\in\{\text{DRS},\text{FBS},\text{FBFS}, \text{DYS}\}$, initial value $z^0\in \Real^{2p}$, regularization coefficient $\eta>0$, safeguarding constants $D,\epsilon>0$, maximum memory size $M\in \mathbb{Z}_{+}$.

\STATE $z^1_{\text{OS}}=z^1=F_{\text{OS}}(z^0)$, $g^0 = z^0 - z^1$. 
\STATE Initialize $i=0$, $Z_k=z^1_{\text{OS}}$.
\FOR{$k = 1, 2, \dots$}
\STATE $M^k = \min\{k,M\}$.
\IF{$\text{OS}\in\{\text{DRS},\text{FBS}$, \text{DYS}\}}
    \STATE $z^{k+1}_{\text{OS}}=F_{\text{OS}}(z^k)$.
\ELSE
    \STATE $z^{k+1}_{\text{F}} = (\mathcal{I}-\mu \mathcal{P})z^k$, 
    \STATE $z^{k+1}_{\text{FBS}}$ = $\mathcal{J}_{\mu\mathcal{Q}} z^{k+1}_{\text{F}}$,
    \STATE $g^k_{\text{FBS}} = z^k - z^{k+1}_{\text{FBS}}$, \STATE $z^{k+1}_{\text{OS}} = z^k - z^{k+1}_{\text{F}} + (\mathcal{I}-\mu \mathcal{P})z^{k+1}_{\text{FBS}}$.
\ENDIF
\STATE $g^k = z^k - z^{k+1}_{\text{OS}}$.
\STATE  Update $Z_k$ with $z^{k+1}_{\text{OS}}$, $Y_k$ with $y^{k-1} = g^k - g^{k-1}$.
\STATE Update $S_k$ with $s^{k-1}=z^k-z^{k-1}$.
\STATE Compute $\zeta^k\in \Real^{M^k}$ with \Eqn{gammakformula}.
\STATE Compute $\alpha^k\in \Real^{M^k+1}$ with \Eqn{alphakformula}. 
\STATE Compute the AA candidate $z^{k+1}_{\text{AA}}=Z_k \alpha^k$.
\IF{$\text{OS}\in\{\text{DRS},\text{FBS},\text{DYS}\}$ and \\$\lVert g^k \rVert_2\le D\lVert g^0 \rVert_2(i+1)^{-1-\epsilon}$}
    \STATE $z^{k+1} =z^{k+1}_{\text{AA}}$, $i=i+1$.
\ELSIF{$\text{OS}=\text{FBFS}$ and \\$\lVert g^k_{\text{FBS}} \rVert_2\le \frac{1}{2}D\lVert g^0 \rVert_2(i+1)^{-1-\epsilon}$} 
    \STATE $z^{k+1} =z^{k+1}_{\text{AA}}$, $i=i+1$.
\ELSE
    \STATE $z^{k+1} =z^{k+1}_{\text{OS}}$.
\ENDIF
\ENDFOR

\end{algorithmic}
\end{algorithm}

\begin{remark} Although \Alg{a2os} involves multiple hyperparameters, such as $\eta$, $D$, and $\epsilon$, the safeguard steps in \Alg{a2os} are in fact quite lenient due to the typically large value of $\lVert g^0\rVert_2$. In practice, the safeguarding checks in lines 19 and 21 are violated only a few times throughout the iterative process. Hence, these hyperparameters have relatively small impact on the practical performance of the algorithm in most cases, as long as they are conservatively chosen. The memory parameter $M$ is common to all limited-memory AA algorithms. A larger $M$ will enables convergence in potentially fewer iterations but incurs a higher overhead in storing $Y_k, Z_k$ and computing $\zeta^k$ (note that we in fact do not need to store $S_k$). In practice, $5\le M\le 20$ strikes a good balance. We use $M=10$ in our experiments.
\end{remark}
\begin{remark}
If $\eta=0$ and $D=\infty$, the safeguard checks in lines 19 and 21 will always pass, and \Alg{a2os} reduces to standard AA \Eqn{aa2}. If $D=0$, the safeguard checks always fail, and the algorithm reduces to original fixed-point iterations. Mai and Johansson \cite{mai2020anderson} presented a  counterexample to demonstrate that AA can become stuck in a periodic orbit even if the original FPI is globally convergent. We show that our regularized and safeguarded \Alg{a2os} is able to converge on this example in \App{patho}. This provides important numerical evidence for our global convergence result.
\end{remark}
We establish the global convergence of \Alg{a2os} with the following result.
\begin{theorem}\label{thm:aaconverge}
\textit{
Suppose that $\text{zer}\;(\mathcal{P}+\mathcal{Q})$ or $\text{zer}\;(\mathcal{P}+\mathcal{Q}+\mathcal{R})$ is nonempty. Let $z^0 \in \mathbb{R}^{2p}$ be an arbitrary initialization, $\eta$, $D$, and $\epsilon$ be positive hyperparameters, along with a positive integer $M$, then \Alg{a2os} generates a sequence ${z^k}$ that converges to a point $z^*\in \mathbb{R}^{2p}$. For FBS and FBFS, $z^* \in \text{zer}\;(\mathcal{P}+\mathcal{Q})$. For DRS and DYS, the desired optimal points in $\text{zer}\;(\mathcal{P}+\mathcal{Q})$, $\text{zer}\;(\mathcal{P}+\mathcal{Q}+\mathcal{R})$ can be obtained by $\mathcal{\mathcal{J}_{\mu\mathcal{P}}}z^*$ and
$\mathcal{\mathcal{J}_{\mu\mathcal{R}}}z^*$, respectively.}
\end{theorem}

\section{Numerical Experiments}\label{sec:numexp}
In this section, we present a series of numerical examples to illustrate the acceleration capabilities of the proposed algorithm. In all examples, we first introduce the problem and assess the statistical benefits of using CNC regularizer over alternative regularizers to motivate accelerating solvers of \eqref{eq:cost}. Unless stated otherwise, we use the following set of hyperparameters for our experiments: the maximum memory size for AA is  $M=10$, the regularization coefficient is $\eta=10^{-2}$, and the safeguarding constants are  $D=10$ and $\epsilon=10^{-6}$. We terminate the algorithm if $\lVert g^k \rVert_2<(\lVert z^k \rVert_2+1) \epsilon_{\text{tol}}$, where $\epsilon_{\text{tol}}$ is the termination tolerance and by default $\epsilon_{\text{tol}} = 10^{-5}$.  All experiments are completed on a cluster using 8 processor cores with 8GB of RAM for each processor core. Our experiments are fully reproducible with code available at \url{https://github.com/qhengncsu/AA_CNC}. 
\subsection{Sparse Linear Regression}
\label{sec:sparse-linear} 
In this section, we apply the CNC regularization strategy \Eqn{cost} to sparse linear regression. 
We apply the GMC model in \cite{selesnick2017sparse} and the group GMC model in \cite{liu2023convex} for individual and grouped variable selection, respectively. That is, $y\in \Real^n$ \Eqn{cost} is the response vector, $A\in \Real^{n\times p}$ is the design matrix, and $x\in \Real^p$ is the vector of regression coefficients to estimate, which is assumed to have a sparse structure. For individual variable selection, the convex regularizer is  $\rho(x)=\lVert x\rVert_1$. For grouped variable selection, $\rho(x)$ is the $\ell_{2,1}$-norm defined as
\begin{equation}\label{eq:groupnorm}
\rho(x) = \sum_{j=1}^J \sqrt{p_j} \lVert x_{(j)} \rVert_2,
\end{equation}
where $x_{(j)}\in \Real^{p_j}$ is the subvector of $x$ corresponding to the $j$-th group of covariates. Note that $x = (x_{(1)}^\top ,x_{(2)}^\top ,\dots,x_{(J)}^\top )^\top $ and $\sum_{j=1}^J p_j = p$.

{\bf Problem Instance.} We consider a simulated data matrix with $n=2000$ samples and $p = 10000$ predictors.  Each row of the  design matrix $A \in \Real^{n \times p}$ is an i.i.d. draw from the multivariate normal distribution $N_p(0, \Sigma)$, where $\Sigma \in \Real^{p \times p}$'s $ij$-th entry $\sigma_{ij} = 0.3^{|i-j|}$. The true coefficient vector $x^* \in \Real^p$ has its first $50$ components equal to $1$, second $50$ components equal to $-1$, and the remaining entries equal to $0$. The response vector is then generated by $y = A x^* +  \epsilon$, where $\epsilon$ is the Gaussian noise with each component $\epsilon_i$ being an i.i.d. drawn from $\mathcal{N}(0, {x^*}\Tra \Sigma x^*)$ such that the signal-to-noise ratio is $1$. For grouped variable selection, we treat every $50$ consecutive predictors as one group.


\begin{figure*}[!t]
  \centering
  \includegraphics[width=0.9\textwidth]{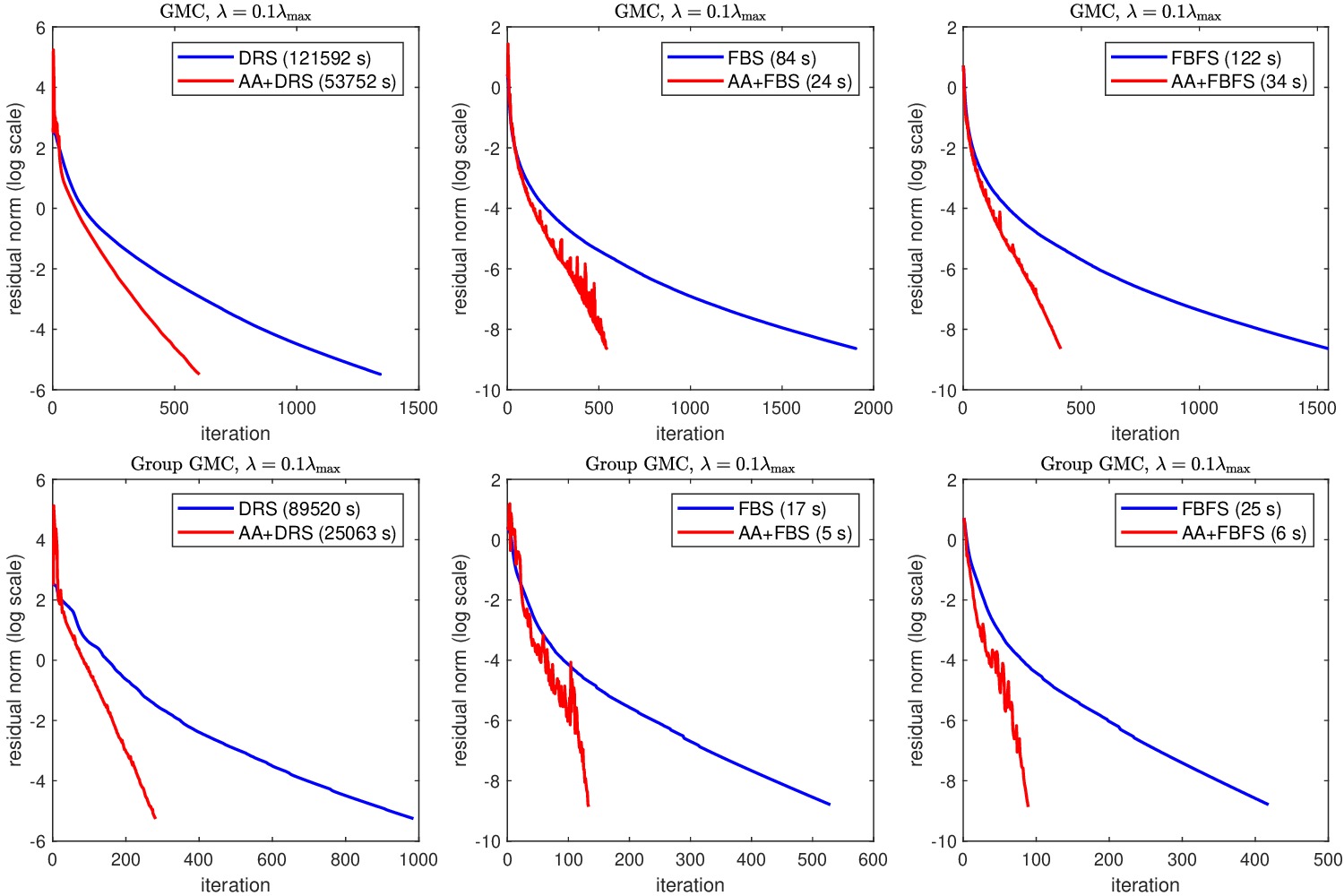}
  \caption{The residual norm trajectories of DRS, FBS, and FBFS in solving the GMC (top row) and the group GMC (bottom row) problems when $\lambda=0.1\lambda_{\max}$. Wall clock run times in seconds in parentheses next to method names in the legend boxes.}
  \label{fig:slr}
\end{figure*} 

{\bf Parameter Setting.} The convexity-preserving parameter $\gamma$ is set at $0.8$. We then apply DRS, FBS, FBFS, and their AA versions to solve the resulting GMC and group GMC problems. For DRS, we use the step size $\mu=0.01$; for FBS, we use the step size $\mu=1.99\times \min\{(1-\gamma)/\gamma,1\}/\lVert A\rVert_2^2$; for FBFS, we use the step size $0.99/(\lVert \begin{bmatrix} 1-\gamma & \gamma \\ -\gamma & \gamma\end{bmatrix} \rVert_2\lVert  A\rVert_2^2)$.

{\bf Statistical Performance.} We first demonstrate the promising statistical performance of the CNC regularization strategy for sparse linear regression and individual variable selection, we compare several accuracy metrics of GMC to three other commonly used penalized sparse estimation methods in statistics, including Lasso \cite{tibshirani1996regression}, SCAD \cite{fan2001variable}, and MCP \cite{zhang2010nearly}. We note that SCAD and MCP are folded concave penalties and the resulting optimization problems are nonconvex. Here we still use the same data generation protocol. We then employ 5-fold cross-validation to select the best regularization parameter for each method and run the experiments for $20$ replicates. We evaluate the accuracy of each method in the following aspects: (i) estimation error defined as $\lVert \hat{x} - x^* \lVert_2$; (ii)  prediction error defined as $\lVert A \hat{x} - A x^* \lVert_2$; (iii) support identification performance measured by the F1 score which falls in the range of $[0,1]$. A higher F1 score implies better variable selection performance. \Tab{cv} summarizes the results. GMC achieves the highest F1 score and lowest  prediction error. Its estimation error is also comparable to the Lasso's. For grouped variable selection, we refer users to \cite{liu2023convex} for a diverse array of experiments comparing group GMC with other popular grouped variable selection methods.

\begin{table}
\centering
\caption{Accuracy metrics of different sparse estimation methods}
\begin{tabular}{cccc}
\hline
  \\[-2ex]
Method &   Estimation error  & Prediction error & F1 score\\[0.5ex]
 \hline
   \\[-2ex]
 \multirow{1}{*}{Lasso}& $5.97$  & $275.33$  & $0.44$  \\
  \hline
    \\[-2ex]
 \multirow{1}{*}{SCAD} & $6.61$  & $298.72$ & $0.57$ \\
  \hline
   \multirow{1}{*}{MCP} & $7.22$  & $303.63$ & $0.68$ \\
  \hline
   \multirow{1}{*}{GMC} & $6.17$  & $247.74$ & $0.76$ \\
  \hline
\end{tabular}

\label{tab:cv}
\end{table}

{\bf Convergence Speed.} In \Fig{slr}, we visualize the residual norm trajectories of the compared methods. 
DRS is almost prohibitively slow compared to FBS and FBFS. This is because evaluating the resolvent of $\mathcal{P}$ in \eqref{eq:operators} involves solving a large-scale linear system (dimension is $2p=2\times 10^4$ in this example), which makes the computational cost per iteration of DRS dramatically higher than that of FBS and FBFS. For FBS and FBFS, AA helps reduce the number of iterations required approximately by a factor of $4$. Compared with FBS, FBFS is overall slower since it requires an additional forward step per iteration.

\begin{table}
\centering
\caption{Average computation time (in seconds) of a solution path for sparse linear regression
}
\begin{tabular}{cccc}
\hline
  \\[-2ex]
Problem &  Time (FBS) &  Time (FBS+AA)  & speed-up ratio\\[0.5ex]
 \hline
   \\[-2ex]
 \multirow{1}{*}{GMC}& $2128$  & $439$ & $4.85$\\
  \hline
    \\[-2ex]
 \multirow{1}{*}{Group GMC} & $592$  & $81$ & $7.31$\\
  \hline
\end{tabular}
\label{tab:path}
\end{table}

\Fig{slr} illustrates the numerical performance of \Alg{a2os} at a single value of $\lambda$. In practice, since the best choice of $\lambda$ is not known a priori, one may be more interested in computing the whole solution path. Given that FBS is the fastest splitting method in \Fig{slr}, we now focus on comparing the original FBS and Anderson accelerated FBS in computing the solution path of GMC and group GMC. The regularization parameter $\lambda$ takes on a sequence of $100$ values logarithmically spaced from $\lambda_{\max}$ to $0.001\lambda_{\max}$. Here $\lambda_{\max}$ is the smallest value of $\lambda$ such that \Eqn{cost} yields a zero solution. When $\rho(x) = \|x\|_1$, $\lambda_{\max} = \max_j \{|a_j^\top y|\}$, where $a_j$ is the $j$-th column of the design matrix $A$. When $\rho(x)$ is the $\ell_{2,1}$ norm, $\lambda_{\max} = \max_j \{\|A_{.j}^\top y\|_2/\sqrt{p_j}\}$, where $A_{.j}$ is the submatrix of $A$ whose columns correspond to
the variables in the $j$-th group. When computing the solution path, we initialize, or warm start, the algorithm with the obtained solution $(\hat{x},\hat{v})$ at the previous value of $\lambda$. \Tab{path} reports the average computation time over $20$ random replicates, using the previously described data generation protocol. FBS accelerated by AA is $4.85$ times faster 
than standard FBS for GMC and is $7.31$ times faster for group GMC.

\subsection{Regularized Matrix Regression/Completion}
Matrix completion \cite{mazumder2010spectral} and regression \cite{zhou2014regularized} are both classic machine learning problems in modern data science. For dealing with matrix data, spectral regularization is often employed to recover a low-rank matrix parameter. A spectral-regularized matrix regression problem can be posed as the following optimization problem.
\begin{equation}
\label{eq:Matreg}
    \underset{X\in \Real^{d_1\times d_2}}{\min}\quad \frac{1}{2} \sum_{i=1}^n \left(y_i - \langle A_i, X \rangle\right)^2 + \lambda \lVert X\rVert_*,
\end{equation}
where $n$ is the sample size, $y_i$ is the observed response, $A_i$ is the $i$-th matrix covariate, and $X$ is the matrix parameter to be estimated. Matrix completion is a special case of \Eqn{Matreg}. To recover the classic soft-Impute problem \cite{mazumder2010spectral}, we only need to consider $y_i$ as the observed entries of a matrix and $A_i$ to be a binary $d_1\times d_2$ matrix that selects the corresponding entry. For a convex-nonconvex version of the spectral-regularized matrix regression/completion problem, the convex regularizer is $\rho(X) = \lVert X\rVert_*$. The data fidelity term can be rewritten as $\frac{1}{2} \lVert y-A x\rVert_2^2$, where the $i$-th row of $A \in \Real^{n \times d_1d_2}$ is the transpose of the column-major vectorization of $A_i$, and $x$ is the column-major vectorization of $X$. 

{\bf Problem Instance.} We consider synthetic data to compare the performance of the traditional spectral-regularized matrix regression/completion model and its convex-nonconvex counter parts.  For matrix regression, the sample size is set to $n=1000$, and each covariate matrix $A_i$ is of size $64 \times 64$ with independent standard normal entries. We test two different matrix coefficients: the first true coefficient $X \in \Real^{64 \times 64}$ is binary and forms a cross shape as shown in the left panel of \Fig{X}; the second true coefficient matrix  $X \in \Real^{64 \times 64}$ is a checkerboard generated by the \texttt{checkerboard} function in Matlab and is shown in the right panel of \Fig{X}. The response is then simulated by $y_i = \langle A_i, X\rangle + \epsilon_i$, where $\epsilon_i$ is a standard normal noise. For matrix completion, we use the same low-rank matrix patterns but magnify the coefficient matrices 4-fold to become $256\times 256$. We then add standard normal noise such that the SNR is $1$, and mask $80\%$ of the entries.

\begin{figure}[htbp]
\centering
\includegraphics[width=1.7in]{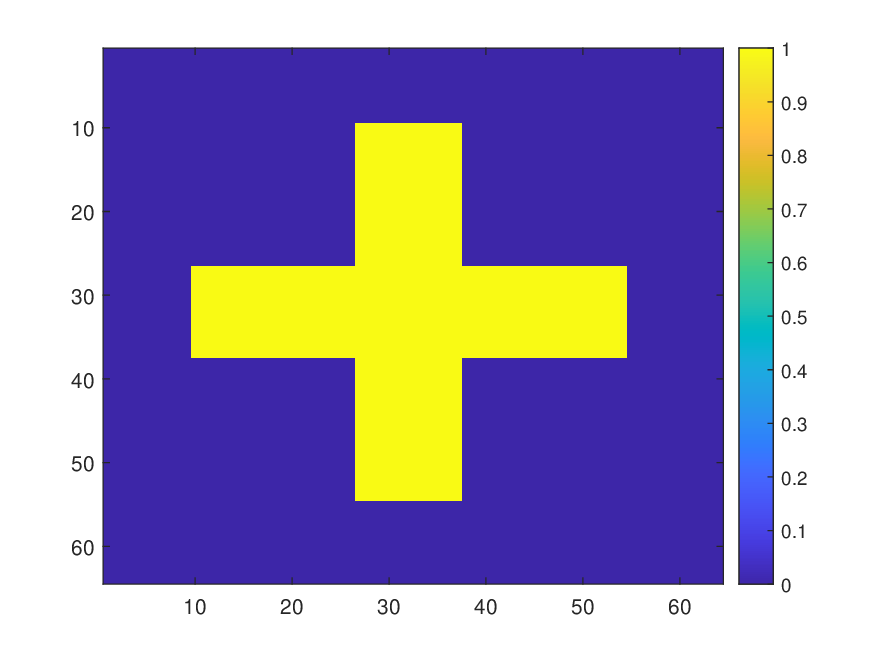}
\includegraphics[width=1.7in]{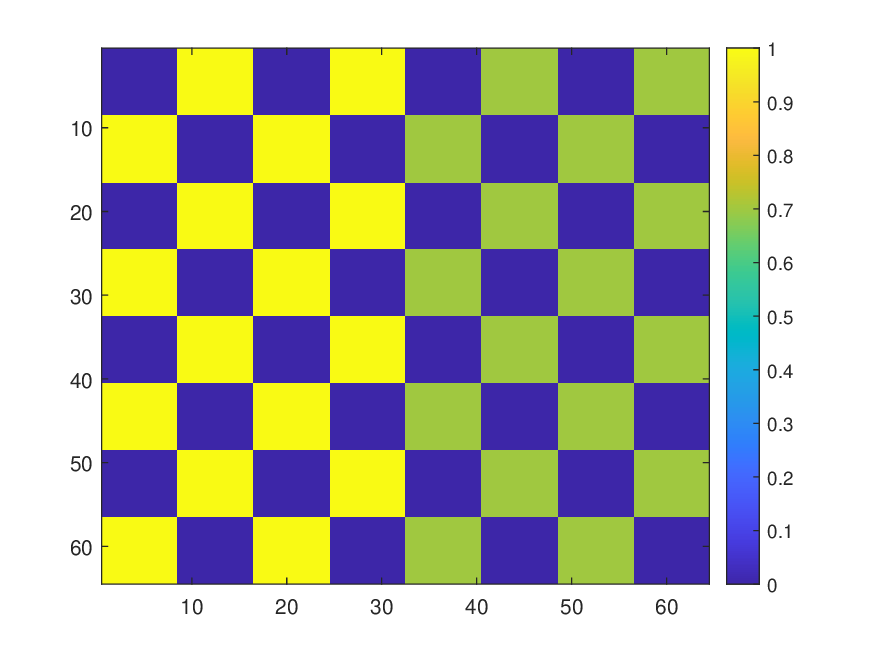}
\caption{A binary coefficient matrix in a cross shape (left panel) and a checkerboard coefficient matrix generated by the \texttt{checkerboard} function in Matlab (right panel).}
\label{fig:X}
\end{figure}

{\bf Parameter Setting.} The convexity-preserving parameter $\gamma$ is set at $0.8$. We apply FBS, FBFS, and their AA versions to solve the resulting CNC regularized matrix regression/completion problems. Though applicable, DR and its AA version are not included in the comparison due to their prohibitively  slow computational speed. For FBS, we again use the step size $\mu=1.99\times \min\{(1-\gamma)/\gamma,1\}/\lVert A\rVert_2^2$; for FBFS, we again set the step size at $0.99/(\lVert \begin{bmatrix} 1-\gamma & \gamma \\ -\gamma & \gamma\end{bmatrix} \rVert_2\lVert  A\rVert_2^2)$. Notice that in the case of matrix completion, we have $\lVert  A\rVert_2=1$. 

For each method, we compute a solution path with $\lambda$ taking on a sequence of $20$ values  logarithmically spaced between $10^{-1}$
to $10^3$. We employ BIC as provided in \cite{zhou2014regularized} for selecting the regularization parameter $\lambda$. To evaluate different methods, we consider the prediction performance measured by the prediction error $\lVert Ax - A\hat{x} \rVert_2$, where $\hat{x}$ is the vectorized version of the estimated coefficient matrix $\hat{X}$.

{\bf Statistical Performance.} Table \ref{tab:MatReg_perf} compares the prediction performance of the original spectral regularization and the proposed CNC regularization for matrix regression/completion under two different scenarios. 
CNC regularization consistently achieves higher accuracy than convex spectral regularization. We note that to obtain the the results of convex spectral regularization, we simply need to set $\gamma=0$ in our algorithm.

\begin{table}[htbp]
\centering
\caption{Accuracy metrics of different regularization methods for matrix regression/completion.}
\label{tab:MatReg_perf}
\begin{tabular}{cccc} 
\hline
Signal                        & Problem                     & Method & Prediction Error  \\ 
\hline
\multirow{4}{*}{~Cross}       & \multirow{2}{*}{Regression} & Convex &         52.82          \\  
                              &                             & CNC    &        16.12           \\ 
\cline{2-4}
                              & \multirow{2}{*}{Completion} & Convex &          43.84         \\
                              &                             & CNC    &         32.50          \\ 
\hline
\multirow{4}{*}{Checkerboard} & \multirow{2}{*}{Regression} & Convex &           50.27        \\
                              &                             & CNC    &        15.97          \\ 
\cline{2-4}
                              & \multirow{2}{*}{Completion} & Convex &            48.14       \\
                              &                             & CNC    &           33.01   \\
\hline
\end{tabular}
\end{table}

\begin{table}[htbp]
\centering
\caption{Average computation time (in seconds) of a solution path for regularized matrix regression/completion.}
\label{tab:MatReg_path}
\resizebox{0.45\textwidth}{!}{
\begin{tabular}{cccccc} 
\hline
Signal                        & Problem                     & Algorithm & Original & AA  & speed-up ratio  \\ 
\hline
\multirow{4}{*}{~Cross}       & \multirow{2}{*}{Regression} & FBS       & 585      & 269 & 2.17            \\
                              &                             & FBFS      & 817      & 334 & 2.45            \\ 
\cline{2-6}
                              & \multirow{2}{*}{Completion} & FBS       & 415      & 106 & 3.92            \\
                              &                             & FBFS      & 349      & 87 & 4.01            \\ 
\hline
\multirow{4}{*}{Checkerboard} & \multirow{2}{*}{Regression} & FBS       &   557       &   248  &    2.25             \\
                              &                             & FBFS      &     803     &   320  & 2.51                \\ 
\cline{2-6}
                              & \multirow{2}{*}{Completion} & FBS       &  417    & 137 &  3.04               \\
                              &                             & FBS       &   353   & 84  &   4.2              \\
\hline
\end{tabular}}
\end{table}


{\bf Convergence Speed.} Table \ref{tab:MatReg_path} summarizes the average computational time of each algorithm to solve the CNC regularized matrix regression/completion problems over 20 random replicates. We observe an interesting trade-off between FBS and FBFS. Recall that FBS only requires one forward step per iteration while FBFS requires two. However, FBFS allows for a larger step size at $\gamma=0.8$ and is able to converge in fewer iterations. See \Fig{step} for an illustration. Therefore, when the computational cost per iteration is denominated by the backward step, FBFS can out perform FBS. This intuition is corroborated in \Tab{MatReg_path}. For matrix regression, FBS is faster than FBFS since the forward step is more expensive than the backward step. For matrix completion, FBFS in fact outperforms FBS since the singular value thresholding step in the backward iteration is the main computational bottleneck. For both FBS and FBFS, applying AA can speed up the computation by a factor of 2 to 4 on the considered problems.  

\begin{figure}[htbp]
\centering
\includegraphics[width=0.5\textwidth]{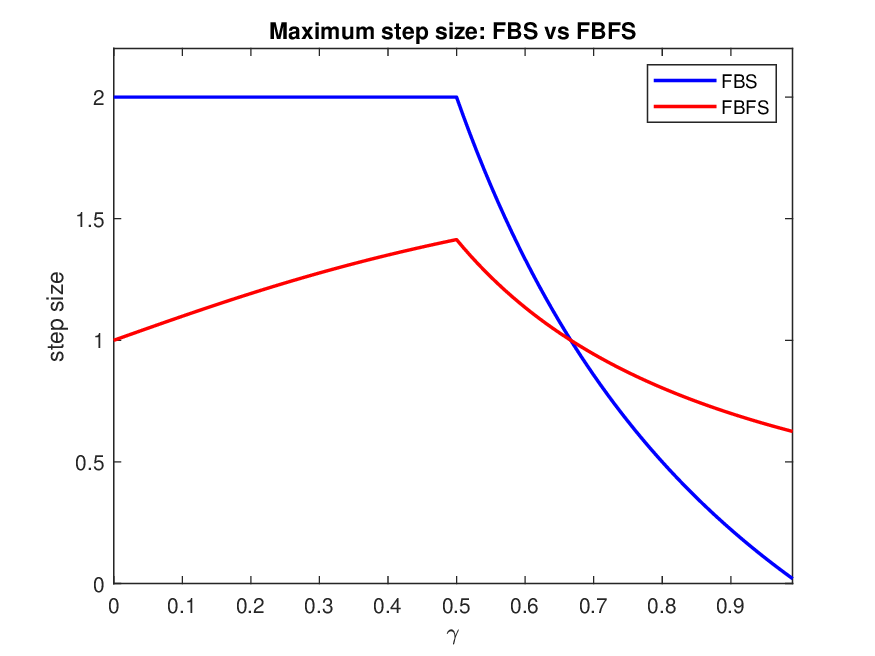}
\caption{Maximum step size of FBS and FBFS when $\lVert A \rVert_2=1$. }
\label{fig:step}
\end{figure}

\subsection{Sparse Group Lasso}\label{sec:sgl}
Sparse group lasso is an extension of the Lasso regression technique that incorporates sparsity at both the individual feature level and group level \cite{simon2013sparse}. This method is particularly useful in scenarios where predictors can be naturally grouped, enabling the selection of entire relevant groups of features while simultaneously enforcing sparsity within those groups. For a convex-nonconvex version of sparse group lasso, we operate within the framework of \Eqn{twominimax}. The first convex regularizer is $\rho_1(x)=\lVert x\rVert_1$. The second convex regularizer $\rho_2(x)$ is the $\ell_{2,1}$-norm defined in \Eqn{groupnorm}. Following standard practice for sparse group lasso problems, we maintain a fixed ratio in  
the two regularization parameters  by setting $\lambda_1 = \lambda, \lambda_2 = \alpha\lambda$. We use $\alpha=\frac{1}{19}$ in this section. 

{\bf Problem Instance.} We consider a DNA methylation data set \cite{klosa2020seagull}. The data set contains blood DNA methylation profiles at about 1.9 million CpG sites of 141 mice and their chronological age (in months). We use the following data preprocessing workflow, which has modest differences with the one in  \cite{klosa2020seagull}.  First, the variable groups are obtained using a singular value decomposition of the full data matrix. We then use sure-independence variable screening \cite{fan2008sure} to downselect the predictor set to a more manageable pool of 10,000 covariates. Next, the data set is split into a training set ($n=75$) and a validation set ($n=66$). 
 The standardized design matrix is denoted as $\tilde{A}_{\text{train}}$ and the centered response is denoted by $\tilde{y}_{\text{train}}$.
The column standard deviations are denoted by $s_{\text{train}}$. Let $\tilde{x}$ denote the coefficient estimate obtained from applying a sparse group lasso model to $(\tilde{A}_{\text{train}},\tilde{y}_{\text{train}})$. The chronological age in the validation set is predicted as $\hat{y}_{\text{val}} = A_{\text{val}}(\tilde{x}\oslash s_{\text{train}})$, where $\oslash$ denotes elementwise division. We measure the goodness of fit using $R^2_{\text{val}} = \text{corr}^2(\hat{y}_{\text{val}},y_{\text{val}})$.

{\bf Parameter Setting.} We use $\gamma=0$ and $0.8$. Again, when $\gamma=0$, \Eqn{twominimax} reduces to the standard convex sparse group lasso. For each value of $\gamma$, we compute a solution path of length 100 starting from $\lambda = 10^{-0.2}\lambda_{\max}$ and ending at $\lambda = 10^{-3}\lambda_{\max}$, with $\lambda$ values logarithmically spaced and $\lambda_{\max}=\max_j \{|a_j^\top y|\}$. We consider only DYS in this section since it is the only splitting method capable of handling two convex regularizers. We set the step size at $\mu=1.99\times \min\{(1-\gamma)/\gamma,1\}/\lVert  \tilde{A}_{\text{train}}\rVert_2^2$.
\begin{figure}[htbp]
\centering
\includegraphics[width=0.5\textwidth]{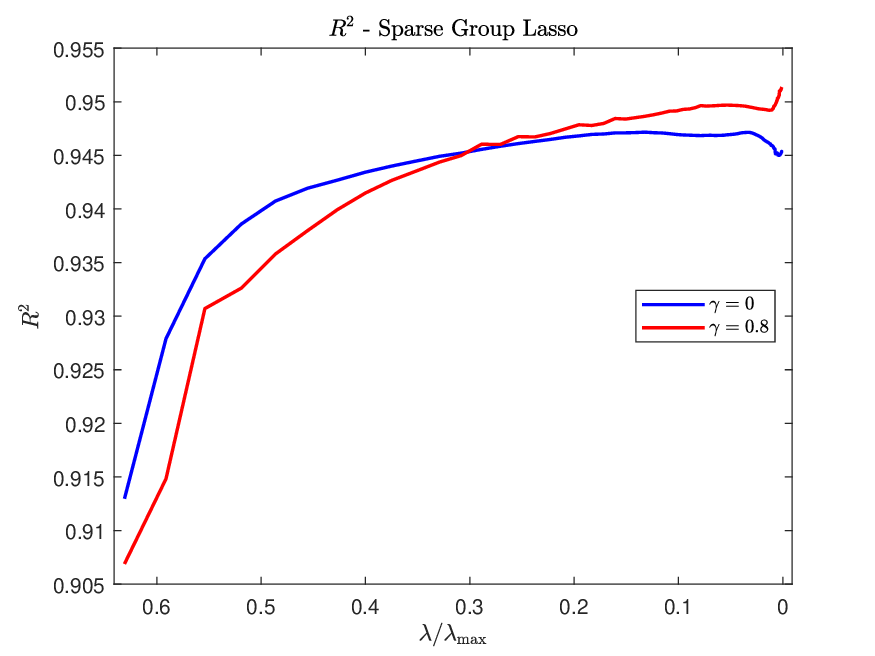}
\caption{$R^2$ on the validation set as $\lambda/\lambda_{\max}$ varies for sparse group lasso. }
\label{fig:R2sgl}
\end{figure}

\begin{table}[h]
\centering
\caption{Time for computing the solution path for sparse group lasso using DYS, with and without AA.}
\begin{tabular}{cccc}
\hline
  \\[-2ex]
$\gamma$ &  Original &  AA  & speed-up ratio\\[0.5ex]
 \hline
   \\[-2ex]
0 & $163$  & $43$ & $3.79$\\
  \hline
    \\[-2ex]
0.8 & $251$  & $130$ & $1.93$\\
  \hline
\end{tabular}
\label{tab:timesgl}
\end{table}

{\bf Statistical Performance.} \Fig{R2sgl} shows how the $R^2$ computed on the validation data set varies as $\lambda/\lambda_{\max}$ varies from $10^{-0.2}$ to $10^{-3}$. We see that in this problem, prediction power is generally better when the penalty parameter is relatively small. From the right side of \Fig{R2sgl}, we see that the best attainable $R^2$ of CNC ($\gamma=0.8$) is modestly better than its convex counterpart ($\gamma=0$). The improvement may be attributed to the fact that CNC induces less shrinkage bias on the the regression coefficients. 

{\bf Convergence Speed.} In terms of computation time, AA achieves a $3.79$-fold speed-up for the convex case ($\gamma=0$) and a $1.93$-fold speed up for the convex-nonconvex case ($\gamma=0.8$). In our experiments, we observe that for this problem, the safeguarding steps appears to be invoked much more often than in the previous experiments.  \Fig{dys_resnorm} shows a plot comparing the residual norm trajectories of the original DYS iterations, naive AA \Eqn{aa2}, and our regularized and safeguarded AA (\Alg{a2os}) with $\lambda = 10^{-0.2} \lambda_{\max}$.  \Alg{a2os} outperforms naive AA, which suggests that the regularization and safeguarding steps not only ensures the theoretical global convergence but also improves the convergence performance in practice.
\begin{figure}[htbp]
\centering
\includegraphics[width=0.5\textwidth]{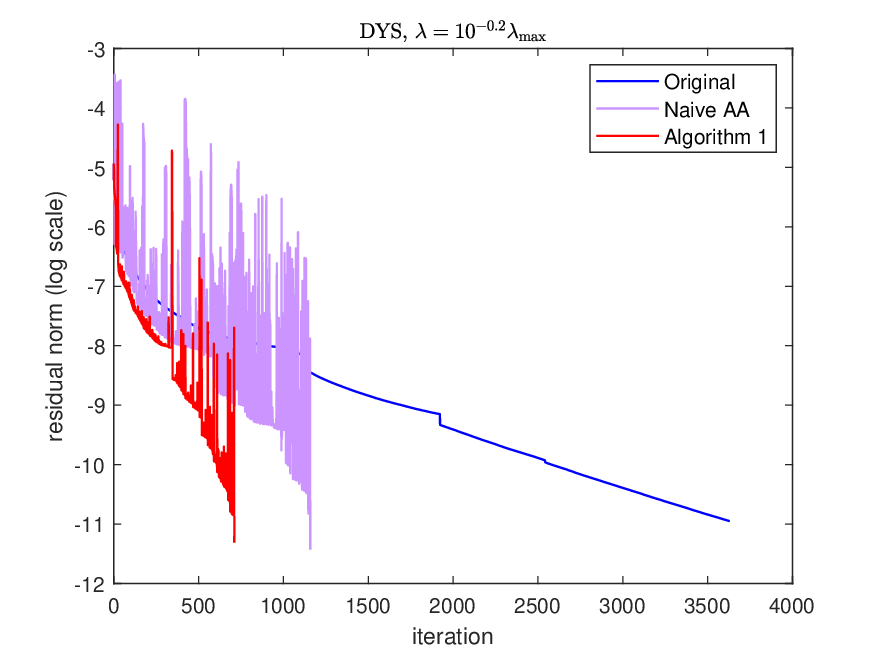}
\caption{The residual norm trajectories of DYS using the original FPI, naive AA, and regularized and safeguarded AA (Algorithm 1) at $\lambda = 10^{-0.2} \lambda_{\max}$. }
\label{fig:dys_resnorm}
\end{figure}
\section{Discussion}\label{sec:discuss}

Motivated by the wide application of CNC regularization but limited investigation on its computation, in this paper, we proposed a unified framework of Anderson acceleration that covers Douglas-Rachford splitting, forward-backward splitting, forward-backward-forward splitting, and Davis-Yin splitting. This framework proves to be highly effective in solving a diverse set of CNC problems, such as sparse linear regression, spectral regularized matrix regression/completion, and sparse group lasso.  Simulation studies and real data applications showed that  the proposed framework was able to significantly speed up the computation of these CNC problems. The proposed accelerated splitting methods come with global convergence guarantees to global minima, which is a key advantage of CNC regularization. Moreover, the assumptions required to establish global convergence are not specific to CNC regularization. Consequently, the proposed framework applies to a much broader range of problems.

There are many potential directions for future investigation. First, the convergence result of the proposed acceleration framework relies on the averagedness property of the FPI. It would be interesting to explore whether this framework can be extended to FPIs that do not possess the averagedness property, for instance, the forward-reflected-backward splitting method proposed by \cite{malitsky2020forward}. Second, in this paper we have focused on problems where the data fidelity term is a least squares objective. Chen et al.\@ \cite{chen2023unified} generalized the CNC framework to include a Poisson response. Their difference-of-convex algorithm, however, requires solving many subproblems iteratively. It is of interest that whether the problem of \cite{chen2023unified} can be placed in an operator splitting/variation inequality framework and solved more efficiently. Additionally, difference-of-convex iterations is a form of Majorization Minorization algorithm, which should also be amenable to acceleration.

\appendix
\subsection{Proof of \Prop{satisfy}}\label{sec:assumptionproof}
\begin{proof}

$\mathcal{P}$ and $\mathcal{Q}$ defined in \Eqn{operators} are maximal monotone since $\mathcal{P}$ is single-valued and $\mathcal{Q}$ is the subdifferential of a convex function. $\mathcal{P}+\mathcal{Q}$ is maximal monotone since $\mathcal{P}$ and $\mathcal{Q}$ are maximal monotone and $\text{dom}\;\mathcal{P} = \Real^{2p}$ (see  \cite[Theorem 24.4 (i)]{bauschke2011convex}).  A proof that $\mathcal{P}$ is $\beta$-coercive can be found in \cite[Proposition 15]{selesnick2017sparse}. Here we fill in some essential details. 

According to \cite[Proposition 15]{selesnick2017sparse}, as long as 
\begin{align}
\begin{split}\label{eq:psd}
\left( \left( \begin{bmatrix}
1-\gamma & 0 \\
0  & \gamma
\end{bmatrix}-\beta_1 \begin{bmatrix}
1-2\gamma+2\gamma^2 & \gamma-2\gamma^2 \\
\gamma-2\gamma^2  & 2\gamma^2
\end{bmatrix}\right)\Kron I_p\right) & \succeq  0, \\
I_p -\beta_2A^\top A & \succeq  0,
\end{split}
\end{align}
then $\mathcal{P}$ is cocoercive. 

The second inequality in \Eqn{psd} is satisfied if $\beta_2\le 1/\lVert  A\rVert_2^2$. The first inequality is satisfied if 
$$
M_{\beta_1} = \begin{bmatrix}
1-\gamma & 0 \\
0  & \gamma
\end{bmatrix}-\beta_1 \begin{bmatrix}
1-2\gamma+2\gamma^2 & \gamma-2\gamma^2 \\
\gamma-2\gamma^2  & 2\gamma^2
\end{bmatrix} \succeq  0 
$$

The $2\times 2$ symmetric real matrix $M_{\beta_1}$ is positive semidefinite if and only if 
\begin{eqnarray}
\label{eq:Mbeta1_inequalities}
\text{trace}(M_{\beta_1}) \geq 0 \quad\text{and}\quad \det(M_{\beta_1}) \ge 0.
\end{eqnarray}

The first inequality in \eqref{eq:Mbeta1_inequalities} is equivalent to
\begin{eqnarray}
\label{eq:bound1}
1 - \beta_1(1-2\gamma+4\gamma^2) \ge 0,
\end{eqnarray}
and the second inequality in \eqref{eq:Mbeta1_inequalities} is equivalent to
\begin{eqnarray}
\label{eq:bound2}
\gamma\beta_1^2-\beta_1+(1-\gamma) \ge 0 .
\end{eqnarray}
The roots of the quadratic function in \eqref{eq:bound2} are
$$
x_1 = \frac{1-|2\gamma-1|}{2\gamma} \quad\text{and} \quad x_2 = \frac{1+|2\gamma-1|}{2\gamma}.
$$
The inequality in \eqref{eq:bound2} is statisfied if $\beta_1\le x_1$. Note that $x_1$ can be written as $x_1 = \min\{1,\frac{1-\gamma}{\gamma}\}$. We next show that the inequality $\beta_1 \le \min\{1,\frac{1-\gamma}{\gamma}\}$ implies $\beta_1\le r(\gamma) = \frac{1}{1-2\gamma+4\gamma^2}$ which is equivalent to the inequality in \eqref{eq:bound1}. Specifically, we show that $r(\gamma) \geq \max \left\{ 1, \frac{1 - \gamma}{\gamma}\right\}$.

It is straightforward to verify that $r(\gamma)^{-1} \leq 1$ when $\gamma \in [0, 0.5)$. Consequently, $r(\gamma) \geq 1$ for $\gamma \in [0, 0.5)$.
When $\gamma \in [0.5, 1]$, we show that 
\begin{eqnarray}
\label{eq:bound3}
r(\gamma) \ge \frac{1-\gamma}{\gamma}.
\end{eqnarray}
The inequality in \eqref{eq:bound3} is equivalent to
\begin{eqnarray}
\label{eq:bound4}
a(\gamma) = 4\gamma^3 -6\gamma^2+4\gamma -1\ge 0.
\end{eqnarray}
Note that
$$
a'(\gamma) = 12\gamma^2 -12\gamma + 4 = 12\left(\gamma-\frac{1}{2}\right)^2 + 1 >0.
$$
Thus, $a(\gamma)$ is increasing on $[0.5, 1]$ and since $a(0.5) = 0$ the inequality in \eqref{eq:bound4} holds. 
Thus $\mathcal{P}$ is $\beta$-cocoercive with $\beta = \beta_1\beta_2 = \frac{\min\left\{1,\frac{1-\gamma}{\gamma}\right\}}{\lVert  A\rVert_2^2}$.

Since $\mathcal{P}$ is an affine operator, it is Lipschitz and the 
the Lipschitz constant of $\mathcal{P}$ is simply the operator norm of $P$. We have
$$
\lVert P \rVert_2 = \left\lVert \begin{bmatrix} 1-\gamma & \gamma \\ -\gamma & \gamma\end{bmatrix} \right\rVert_2\lVert  A\rVert_2^2,
$$
due to the fact that the operator norm of the Kronecker product is the product of operator norm. Thus $\mathcal{P}$ is $L$-Lipschitz with $L=\left\lVert \begin{bmatrix} 1-\gamma & \gamma \\ -\gamma & \gamma\end{bmatrix} \right\rVert_2\lVert  A\rVert_2^2$. 

\end{proof}

\subsection{Proof of \Thm{aaconverge}}\label{sec:convergeproof}

Our proof follows the main logic of \cite[Theorem 4.3]{fu2020anderson}. The theoretical contribution here lies in generalizing \cite[Theorem 4.3]{fu2020anderson} to include several other common splitting methods. In addition, we found and corrected an error in the original safeguarding scheme. More specifically,
\cite{fu2020anderson} invokes the safeguard periodically. We find that in order for the arguments to hold, the safeguard must be excuted for every iteration.

But before we present our proof, we discuss  results and facts the argument uses.

First note that \eqref{eq:alphakformula} implies the following relationship between $z^{k+1}$, $z^{k}$, and $g^k$ in the regularized AA scheme \Eqn{aa2ucreg}:
$$
z^{k+1} = z^k - H_k g^k,
$$
where
$$
H_k = I+(S_k-Y_k)(Y_k^\top Y_k + \eta(\lVert S_k\rVert_{\text{F}}^2+\lVert Y_k\rVert_{\text{F}}^2)I)^{\dagger} Y_k^\top.
$$


\begin{lemma}\label{lem:hk2norm}
\textit{
The matrix $H_k$ satisfies $\lVert H_k\rVert_2\le 1+\frac{2}{\eta}$.}
\end{lemma}

\begin{proof}
See  \cite[Lemma 4.2]{fu2020anderson}. 
\end{proof}
\begin{definition}
A sequence $\{z^k\} \subset \Real^{2p}$ is quasi-Fej\'er monotone with respect to a non-empty target set $C\subset \Real^{2p}$,  if for any $z\in C$, there exisits a nonnegative and summable sequence $\epsilon_k$, such that for any $k\ge 0$, we have
$$
\lVert z^{k+1} - z\rVert_2^2\le \lVert z^{k} - z\rVert_2^2 + \epsilon_k.
$$
\end{definition}

\begin{lemma}\label{lem:quasi}
Let $\{z^k\}$ be a quasi-Fej\'er monotone sequence with respect to an non-empty target set $C\subset \Real^{2p}$, then $\{z^k\}$ converges to a point in $C$ if and only if for any limit point $z$ of $\{z^k\}$, we have $z\in C$. 
\end{lemma}
\begin{proof}
See  \cite[Theorem 3.8]{combettes2001quasi}. 
\end{proof}

We are now ready to prove \Thm{aaconverge}.

\begin{proof}

Let $\ell_j$ denote the iteration where $z^{k+1}_{\text{OS}}$ is adopted the $j$th time. Let $k_i$ denote the iteration where the safeguard check in line 13 or 15 of \Alg{a2os} is passed for the $i$th time. Suppose that after $k-1$ iterations, the unaccelerated iterate has been adopted $j-1$ times and the safeguard check has passed $i-1$ times. Then at the $k$th iteration, either $k=\ell_j$ or $k=k_i$. We first establish that the Anderson acceleration candidate $z^{k+1}_{\text{AA}}$ is adopted  infinitely often. If $z^{k+1}_{\text{AA}}$ is adopted finitely many times, the algorithm reduces to original fixed-point iterations from some point onward. For DRS, FBS, and DYS, this implies $\underset{k\rightarrow \infty}{\lim} \lVert g^{k} \rVert_2 = 0$. For FBFS, this implies $\underset{k\rightarrow \infty}{\lim} \lVert g^{k}_{\text{FBS}} \rVert_2 = 0$ (see \cite[Theorem 25.8 (i)]{bauschke2011convex}). Thus the safeguard checks must eventually pass, and we arrive at a contradiction. 

For DRS, FBS, and DYS, consider any $w\in \Real^{2p}$, since $z^{k+1}_{\text{AA}}$ is adopted at $k_i$, 
\begin{align}
\begin{split}\label{eq:distbd}
\lVert z^{k_i+1} - w\rVert_2  
& \le \lVert z^{k_i} - w\rVert_2+\lVert z^{k_i}-z^{k_i+1} \rVert_2\\
& = \lVert z^{k_i} - w\rVert_2+\lVert H_{k_i}g^{k_i} \rVert_2\\
&\le \lVert z^{k_i} - w\rVert_2 + \left(1+\frac{2}{\eta}\right)\lVert g^{k_i}\rVert_2\\
& \le \lVert z^{k_i} - w\rVert_2+ \left(1+\frac{2}{\eta}\right)D\lVert g^0\rVert_2(i+1)^{-(1+\epsilon)}.
\end{split}
\end{align}
The first inequality is due to triangle inequality. The second inequality is due to \Lem{hk2norm}. The third inequality is due to the fact that the safeguard check is passed. 

For FBFS, $\mathcal{I}-\mu \mathcal{P}$ is 2-Lipschitz due to the $L$-Lipschitzness of $\mathcal{P}$ and the fact that $\mu<1/L$. We can then bound $\lVert g^{k_i} \rVert_2$ as
\begin{align} \label{eq:fbfgkbound}
\begin{split}
\lVert g^{k_i} \rVert_2 & =  \lVert (\mathcal{I}-\mu \mathcal{P})z^{k_i} - (\mathcal{I}-\mu \mathcal{P}) z^{k_i+1}_{\text{FBS}}\rVert_2 \\
& \le 2 \lVert z^{k_i} - z^{k_i+1}_{\text{FBS}}\rVert_2 = 2 \lVert g^{k_i}_{\text{FBS}} \rVert_2.
\end{split}
\end{align}
Thus for FBFS, we similarly have
\begin{align}
\begin{split}\label{eq:distbdfbf}
\lVert z^{k_i+1} - w\rVert_2  
&\le \lVert z^{k_i} - w\rVert_2 + \left(1+\frac{2}{\eta}\right)\lVert g^{k_i}\rVert_2\\
&\le \lVert z^{k_i} - w\rVert_2 + 2\left(1+\frac{2}{\eta}\right)\lVert g^{k_i}_{\text{FBS}}\rVert_2\\
& \le \lVert z^{k_i} - w\rVert_2+ \left(1+\frac{2}{\eta}\right)D\lVert g^0\rVert_2(i+1)^{-(1+\epsilon)}.
\end{split}
\end{align}

 Now, for DRS, consider $z^* \in \text{Fix} \; F_{\text{DRS}}$, by \cite[Proposition 4.25 (iii)]{bauschke2011convex}, we have
\begin{equation}
\label{eq:distbd2dr}
    \lVert z^{l_{j}+1} - z^*\rVert_2^2 \le \lVert z^{l_{j}} - z^*\rVert_2^2 -  \lVert g^{l_j} \rVert_2^2\le \lVert z^{l_{j}} - z^*\rVert_2^2.
\end{equation} 

For FBS, according to  \cite[Theorem 25.8]{bauschke2011convex}, the FPI $F_{\text{FBS}}$ is $\alpha_1 = \frac{1}{\delta}$ averaged, where $\delta=\min\{1,\beta/\mu\}+1/2>1$. Consider $z^* \in \text{Fix} \; F_{\text{FBS}}$, we similarly have 
\begin{equation}
\label{eq:distbd2fb}
    \lVert z^{l_{j}+1} - z^*\rVert_2^2 \le \lVert z^{l_{j}} - z^*\rVert_2^2 -  \frac{1-\alpha_1}{\alpha_1}\lVert g^{l_j} \rVert_2^2\le \lVert z^{l_{j}} - z^*\rVert_2^2.
\end{equation}

For DYS, consider $z^* \in \text{Fix} \; F_{\text{DYS}}$, due to  \cite[Proposition 3.1]{davis2017three}, we have
\begin{equation}
\label{eq:distbd2dy}
    \lVert z^{l_{j}+1} - z^*\rVert_2^2 \le \lVert z^{l_{j}} - z^*\rVert_2^2 -  \frac{1-\alpha_2}{\alpha_2}\lVert g^{l_j} \rVert_2^2\le \lVert z^{l_{j}} - z^*\rVert_2^2,
\end{equation}
where $\alpha_2 = \frac{2\beta}{4\beta-\mu}<1$.  

For FBFS, consider $z^*\in \text{zer}\;(\mathcal{P}+\mathcal{Q})$, due to \cite[Theorem 25.10]{bauschke2011convex}, we have
\begin{equation}
\label{eq:distbd2fbf}
    \lVert z^{l_{j}+1} - z^*\rVert_2^2 \le \lVert z^{l_{j}} - z^*\rVert_2^2 -  (1-\mu^2L^2)\lVert g^{l_j}_{\text{FBS}} \rVert_2^2\le \lVert z^{l_{j}} - z^*\rVert_2^2.
\end{equation}

Combine \Eqn{distbd2dr}, \Eqn{distbd2fb}, \Eqn{distbd2dy}, \Eqn{distbd2fbf} with \Eqn{distbd}. Then for any $k\ge 0$ and any splitting method, we have
\begin{align}
\begin{split}\label{eq:boundedness}
\lVert z^k - z^* \rVert_2 &\le \lVert z^0 - z^* \rVert_2 + \left(1+\frac{2}{\eta}\right) D\lVert g^0\rVert_2\sum_{i=0}^\infty (i+1)^{-(1+\epsilon)} 
\\ & = E < \infty.
\end{split}
\end{align}

Squaring both sides of \Eqn{distbd} and taking $w=z^*$, we have
\begin{align}
\begin{split}
\label{eq:distbdsq}
&\lVert z^{k_i+1} - z^*\rVert_2^2 - \lVert z^{k_i} - z^*\rVert_2^2 \\&\le 2\lVert z^{k_i} - z^*\rVert_2 \left(1+\frac{2}{\eta}\right)D\lVert g^{0} \rVert_2(i+1)^{-(1+\epsilon)}\\&+\left[\left(1+\frac{2}{\eta}\right)D\lVert g^{0}\rVert_2\right]^2(i+1)^{-2(1+\epsilon)}\\
&\le 2E \left(1+\frac{2}{\eta}\right)D\lVert g^{0} \rVert_2(i+1)^{-(1+\epsilon)} \\&+\left[\left(1+\frac{2}{\eta}\right)D\lVert g^{0}\rVert_2\right]^2(i+1)^{-2(1+\epsilon)},
\end{split}
\end{align}
where the second inequality is due to \Eqn{boundedness}. 

The bound in \Eqn{distbdsq} together with the bounds in \Eqn{distbd2dr}, \Eqn{distbd2fb}, \Eqn{distbd2dy}, and \Eqn{distbd2fbf} imply that the sequence $\{z^k\}$ is quasi-Fej\'er monotone with respect to a non-empty target set $C\subset \Real^{2p}$. For DRS, $C=\text{Fix} \; F_{\text{DRS}}$. For FBS, $C=\text{Fix} \; F_{\text{FBS}}=\text{zer}\;(\mathcal{P}+\mathcal{Q})$ (see  \cite[Proposition 25.1 (iv)]{bauschke2011convex}). For DYS, $C=\text{Fix} \; F_{\text{DYS}}$. This is because for any $z^*\in C$ we have
$$
\lVert z^{k+1} - z^*\rVert_2^2 \le \lVert z^{k} - z^*\rVert_2^2 + \epsilon^k,
$$
where
\begin{align*}
\epsilon^{k_i} &= 2E (1+\frac{2}{\eta})D\lVert g^{0} \rVert_2(i+1)^{-(1+\epsilon)}\\ & +((1+\frac{2}{\eta})D\lVert g^{0}\rVert_2)^2(i+1)^{-2(1+\epsilon)},
\end{align*}
and $\epsilon^{l_j}=0$. Since $\sum_{k=0}^\infty \epsilon^k<\infty$, the quasi-Fej\'er monotone property is satisfied.

We now prove that for DRS, either there are only finitely many $l_j$ or $\underset{j\rightarrow \infty}{\lim} \lVert g^{l_j}\rVert_2 = 0$. With an almost identical argument, we can show that this is also the case for FBS and DYS. The arguments for FBS, DYS will only differ in the sense that the negative constant before $\lVert g^{l_j} \rVert_2^2$ will be different for each one of them. For FBFS, we can similarly show that there are only finitely many $l_j$ or $\underset{k\rightarrow \infty}{\lim} \lVert g^{k}_{\text{FBS}} \rVert_2 = 0$ by replacing the term $\lVert g^{l_j} \rVert_2^2$ with $(1-\mu^2L^2)\lVert g^{l_j}_{\text{FBS}} \rVert_2^2$.

From \Eqn{distbd2dr}, we have
\begin{equation}
\label{eq:distbd2drdiff}
\lVert z^{l_{j}+1} - z^*\rVert_2^2 - \lVert z^{l_{j}} - z^*\rVert_2^2\le  -  \lVert g^{l_j} \rVert_2^2.
\end{equation}

Summing \Eqn{distbdsq} and \Eqn{distbd2drdiff} over all $k_i$ and $l_j$ for which $k_i\le k$ and $l_j \le k$ gives
\begin{align}
\begin{split}\label{eq:sqdiffsum}
& \lVert z^{k+1} - z^*\rVert_2^2 - \lVert z^{0} - z^*\rVert_2^2\\ & \le \sum_{i:k_i\le k}  2E \left(1+\frac{2}{\eta}\right)D\lVert g^{0} \rVert_2(i+1)^{-(1+\epsilon)}\\
& + \sum_{i:k_i\le k} \left[\left(1+\frac{2}{\eta}\right)D\lVert g^{0}\rVert_2\right]^2(i+1)^{-2(1+\epsilon)} 
 - \sum_{j:l_j\le k} \lVert g^{l_j} \rVert_2^2.
\end{split}
\end{align}
Adding $\lVert z^{0} - z^*\rVert_2^2$ to both sides of \eqref{eq:sqdiffsum} and taking $k\rightarrow \infty$ gives
\begin{align}
\begin{split}\label{eq:sqdiffsumlimit}
0 &\le \underset{k\rightarrow \infty}{\lowlim} \lVert z^{k+1} - z^*\rVert_2^2 \le \underset{k\rightarrow \infty}{\uplim} \lVert z^{k+1} - z^*\rVert_2^2 \\& \le  \sum_{i=0}^\infty  2E \left(1+\frac{2}{\eta}\right)D\lVert g^{0} \rVert_2(i+1)^{-(1+\epsilon)}\\
& + \sum_{i=0}^\infty \left[\left(1+\frac{2}{\eta}\right)D\lVert g^{0}\rVert_2\right]^2(i+1)^{-2(1+\epsilon)}\\
& - \sum_{l_j<\infty} \lVert g^{l_j} \rVert_2^2 + \lVert z^{0} - z^*\rVert_2^2.
\end{split}
\end{align}
Since $\sum_{i=0}^\infty  2E (1+\frac{2}{\eta})D\lVert g^{0} \rVert_2(i+1)^{-(1+\epsilon)}$ and $\sum_{i=0}^\infty \left[\left((1+\frac{2}{\eta}\right)D\lVert g^{0}\rVert_2\right]^2(i+1)^{-2(1+\epsilon)}$ are finite, the series $\lVert g^{l_j}\rVert_2^2$ must be summable. Thus, we must have only finitely many $l_j$ or $\underset{j\rightarrow \infty}{\lim} \lVert g^{l_j}\rVert_2 = 0$. We also have $\underset{i\rightarrow \infty}{\lim} \lVert g^{k_i}\rVert_2=0$ since $(i+1)^{-1-\epsilon}\rightarrow 0$. Thus, we have shown that $\underset{k\rightarrow \infty}{\lim} \lVert g^{k} \rVert_2 = 0$.

Now, for DRS, FBS, and DYS, due to the fact that $\underset{k\rightarrow \infty}{\lim} \lVert g^{k} \rVert_2 = 0$, for any convergent subsequence $\{z^{t_k}\}$ of $\{z^k\}$, i.e., $z^{t_k}	\rightarrow z^{**}$, we have
$$
F(z^{**})=F(\underset{k\rightarrow \infty}{\lim} z^{t_k}) = \underset{k\rightarrow \infty}{\lim} F(z^{t_k}) = \underset{k\rightarrow \infty}{\lim} z^{t_k} =  z^{**},
$$ 
where the second equality is due to to the continuity of the FPI and the third equality is due to the fact that $\underset{k\rightarrow \infty}{\lim} \lVert g^{k} \rVert_2 = 0$. For FBFS,  we similarly have $F_{\text{FBS}}(z^{**})=z^{**}$, which implies $z^{**} \in \text{Fix}\;F_{\text{FBS}} = \text{zer}\;(\mathcal{P}+\mathcal{Q})$. 

Finally, we invoke \Lem{quasi} to conclude that $\{z^k\}$ converges to a point $z^*$ in $C$. Notice that for DRS, $\text{zer}\;(\mathcal{P}+\mathcal{Q}) = \mathcal{J}_{\mu \mathcal{P}}(\text{Fix}\;F_{\text{DRS}})$ and for DYS, $\text{zer}\;(\mathcal{P}+\mathcal{Q}+\mathcal{R}) = \mathcal{J}_{\mu \mathcal{R}}(\text{Fix}\;F_{\text{DYS}})$. Thus, the optimal solutions for DRS and DYS are recovered by $\mathcal{J}_{\mu \mathcal{P}} z^*$ and $\mathcal{J}_{\mu \mathcal{R}} z^*$, respectively. 

\end{proof}

\subsection{A Pathological Example}\label{sec:patho}
\begin{figure}[htbp]
  \centering
  \includegraphics[width=0.5\textwidth]{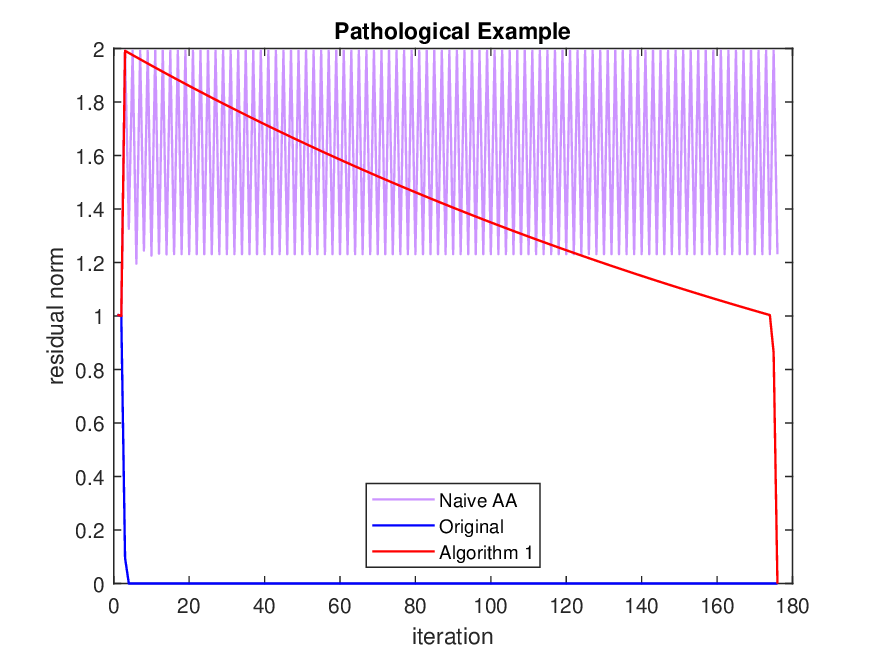}
  \caption{Residual norm plot of the different algorithms applied to the FPI \Eqn{pathfpi} starting from the point $x_0=2.1$.}
  \label{fig:resnormplotpatho}
\end{figure} 
We revisit a pathological example in \cite{mai2020anderson} where an unmodified AA scheme \Eqn{aa2} might fail to converge even if the original FPI is globally convergent. This clarifies the necessity of modifying the algorithm to ensure global convergence. Next, we demonstrate the convergence of our regularized and safeguarded AA algorithm \Alg{a2os} on this pathological example. This numerical demonstration empirically supports our theoretical result on the global convergence of the algorithm.

Consider a smooth convex function $f$ on $\Real$ whose gradient is 
$$
\nabla f(x) = 
\begin{cases}
\frac{x}{10}-24.9 & x\le 1 \\
25x & -1<x<1\\
\frac{x}{10}+24.9 & x\ge 1
\end{cases}.
$$
It is not hard to see that $f$ is $L$-Lipschitz with $L = 25$. We consider the following FPI
\begin{equation}\label{eq:pathfpi}
x^{k+1} = x^k - \frac{1}{25}\nabla f(x^k), 
\end{equation}
to iteratively search for the minimizer of $f$.  This FPI can be viewed as a special case of both FBS \Eqn{fbiteration} and DYS \Eqn{dyiteration}. For AA, we investigate two different parameter settings. In the first setting, we choose $M=1$, $\eta=0$, $D=\infty$, and $\epsilon=0$. Under these parameters, \Alg{a2os} reduces to the naive AA \Eqn{aa2}. In the second setting, we employ the default parameters outlined earlier, with the exception of setting $M=1$ to match the configuration of the naive AA and setting $D=1$ to enable the safeguard to activate earlier. Both algorithm settings are initialized with $x^0=2.1$.

The results shown in \Fig{resnormplotpatho} indicate that in the given example, the original FPI converges quickly, while the naive AA gets trapped in a periodic orbit. Specifically, the iterates cycle over the values $-249(\sqrt{5}-2)$, $249$, $249(\sqrt{5}-2)$, and $-249$, while the residual norms cycle over $1.992$ and $1.231$ after a few iterations. This cyclic behavior occurs whenever the initial iterate $x^0$ falls within the range of $[2.01,246.98]$ (as proven in \cite[Proposition 1]{mai2020anderson}). Our proposed algorithm \Alg{a2os}, initially follows a similar trajectory as the naive AA. However, at iteration 4, the safeguard check enables \Alg{a2os} to escape the cyclic trap. From iteration 4 to iteration 174, the decrease in residual norm is slow because the safeguard check consistently fails and the gradient of $f$ is relatively flat. At iteration 175, the algorithm enters the middle area, after which convergence occurs instantly. Both the original FPI and \Alg{a2os} yield a final solution of 0. Note that although our proposed algorithm is able to converge in this example, it is not faster than the original FPI. However, the value of our proposed algorithm lies in its ability to accelerate convergence in the vast majority of cases, as demonstrated in the \Sec{numexp}.


\section{References}
\bibliographystyle{IEEEtran}
\bibliography{main}

\begin{thebibliography}{10}
\providecommand{\url}[1]{#1}
\csname url@samestyle\endcsname
\providecommand{\newblock}{\relax}
\providecommand{\bibinfo}[2]{#2}
\providecommand{\BIBentrySTDinterwordspacing}{\spaceskip=0pt\relax}
\providecommand{\BIBentryALTinterwordstretchfactor}{4}
\providecommand{\BIBentryALTinterwordspacing}{\spaceskip=\fontdimen2\font plus
\BIBentryALTinterwordstretchfactor\fontdimen3\font minus \fontdimen4\font\relax}
\providecommand{\BIBforeignlanguage}[2]{{%
\expandafter\ifx\csname l@#1\endcsname\relax
\typeout{** WARNING: IEEEtran.bst: No hyphenation pattern has been}%
\typeout{** loaded for the language `#1'. Using the pattern for}%
\typeout{** the default language instead.}%
\else
\language=\csname l@#1\endcsname
\fi
#2}}
\providecommand{\BIBdecl}{\relax}
\BIBdecl

\bibitem{tibshirani1996regression}
R.~Tibshirani, ``Regression shrinkage and selection via the lasso,'' \emph{Journal of the Royal Statistical Society: Series B (Methodological)}, vol.~58, no.~1, pp. 267--288, 1996.

\bibitem{rudin1992nonlinear}
L.~I. Rudin, S.~Osher, and E.~Fatemi, ``Nonlinear total variation based noise removal algorithms,'' \emph{Physica D: Nonlinear Phenomena}, vol.~60, pp. 259--268, 1992.

\bibitem{mazumder2010spectral}
R.~Mazumder, T.~Hastie, and R.~Tibshirani, ``Spectral regularization algorithms for learning large incomplete matrices,'' \emph{The Journal of Machine Learning Research}, vol.~11, pp. 2287--2322, 2010.

\bibitem{fan2001variable}
J.~Fan and R.~Li, ``Variable selection via nonconcave penalized likelihood and its oracle properties,'' \emph{Journal of the American statistical Association}, vol.~96, no. 456, pp. 1348--1360, 2001.

\bibitem{zhang2010nearly}
C.-H. Zhang, ``Nearly unbiased variable selection under minimax concave penalty,'' \emph{Annals of Statistics}, vol.~38, pp. 894--942, 2010.

\bibitem{blake1987visual}
A.~Blake and A.~Zisserman, \emph{Visual reconstruction}.\hskip 1em plus 0.5em minus 0.4em\relax MIT press, 1987.

\bibitem{nikolova1998estimation}
M.~Nikolova, ``Estimation of binary images by minimizing convex criteria,'' in \emph{Proceedings 1998 International Conference on Image Processing. ICIP98 (Cat. No. 98CB36269)}, vol.~2.\hskip 1em plus 0.5em minus 0.4em\relax IEEE, 1998, pp. 108--112.

\bibitem{selesnick2014}
I.~W. Selesnick and I.~Bayram, ``Sparse signal estimation by maximally sparse convex optimization,'' \emph{IEEE Transactions on Signal Processing}, vol.~62, no.~5, pp. 1078--1092, 2014.

\bibitem{selesnick2014convex}
I.~W. Selesnick, A.~Parekh, and I.~Bayram, ``Convex 1-{D} total variation denoising with non-convex regularization,'' \emph{IEEE Signal Processing Letters}, vol.~22, no.~2, pp. 141--144, 2014.

\bibitem{lanza2016convex}
A.~Lanza, S.~Morigi, and F.~Sgallari, ``Convex image denoising via non-convex regularization with parameter selection,'' \emph{Journal of Mathematical Imaging and Vision}, vol.~56, pp. 195--220, 2016.

\bibitem{lanza2017nonconvex}
A.~Lanza, S.~Morigi, I.~Selesnick, and F.~Sgallari, ``Nonconvex nonsmooth optimization via convex-nonconvex majorization-minimization,'' \emph{Numerische Mathematik}, vol. 136, pp. 343--381, 2017.

\bibitem{selesnick2017sparse}
I.~W. Selesnick, ``Sparse regularization via convex analysis,'' \emph{IEEE Transactions on Signal Processing}, vol.~65, pp. 4481--4494, 2017.

\bibitem{lanza2019sparsity}
A.~Lanza, S.~Morigi, I.~W. Selesnick, and F.~Sgallari, ``Sparsity-inducing nonconvex nonseparable regularization for convex image processing,'' \emph{SIAM Journal on Imaging Sciences}, vol.~12, no.~2, pp. 1099--1134, 2019.

\bibitem{selesnick2017total}
I.~Selesnick, ``Total variation denoising via the {M}oreau envelope,'' \emph{IEEE Signal Processing Letters}, vol.~24, no.~2, pp. 216--220, 2017.

\bibitem{yin2019stable}
L.~Yin, A.~Parekh, and I.~Selesnick, ``Stable principal component pursuit via convex analysis,'' \emph{IEEE Transactions on Signal Processing}, vol.~67, no.~10, pp. 2595--2607, 2019.

\bibitem{liu2023convex}
X.~Liu, A.~J. Molstad, and E.~C. Chi, ``A convex-nonconvex strategy for grouped variable selection,'' \emph{Electronic Journal of Statistics}, vol.~17, no.~2, pp. 2912--2961, 2023.

\bibitem{Chen2021}
Y.~Chen, M.~Yamagishi, and I.~Yamada, ``A linearly involved generalized moreau enhancement of $\ell_{2,1}$-norm with application to weighted group sparse classification,'' \emph{Algorithms}, vol.~14, no.~11, 2021.

\bibitem{yuan2006model}
M.~Yuan and Y.~Lin, ``Model selection and estimation in regression with grouped variables,'' \emph{Journal of the Royal Statistical Society: Series B (Statistical Methodology)}, vol.~68, no.~1, pp. 49--67, 2006.

\bibitem{combettes2011proximal}
P.~L. Combettes and J.-C. Pesquet, \emph{Proximal splitting methods in signal processing}.\hskip 1em plus 0.5em minus 0.4em\relax Springer New York, 2011, pp. 185--212.

\bibitem{wang2018nonconvex}
S.~Wang, I.~Selesnick, G.~Cai, Y.~Feng, X.~Sui, and X.~Chen, ``Nonconvex sparse regularization and convex optimization for bearing fault diagnosis,'' \emph{IEEE Transactions on Industrial Electronics}, vol.~65, no.~9, pp. 7332--7342, 2018.

\bibitem{douglas1956numerical}
J.~Douglas and H.~H. Rachford, ``On the numerical solution of heat conduction problems in two and three space variables,'' \emph{Transactions of the American Mathematical Society}, vol.~82, no.~2, pp. 421--439, 1956.

\bibitem{lions1979splitting}
P.-L. Lions and B.~Mercier, ``Splitting algorithms for the sum of two nonlinear operators,'' \emph{SIAM Journal on Numerical Analysis}, vol.~16, no.~6, pp. 964--979, 1979.

\bibitem{eckstein1992douglas}
J.~Eckstein and D.~P. Bertsekas, ``On the {D}ouglas-{R}achford splitting method and the proximal point algorithm for maximal monotone operators,'' \emph{Mathematical Programming}, vol.~55, pp. 293--318, 1992.

\bibitem{tseng2000modified}
P.~Tseng, ``A modified forward-backward splitting method for maximal monotone mappings,'' \emph{SIAM Journal on Control and Optimization}, vol.~38, no.~2, pp. 431--446, 2000.

\bibitem{davis2017three}
D.~Davis and W.~Yin, ``A three-operator splitting scheme and its optimization applications,'' \emph{Set-valued and Variational Analysis}, vol.~25, pp. 829--858, 2017.

\bibitem{Anderson65}
D.~G. Anderson, ``Iterative procedures for nonlinear integral equations,'' \emph{Journal of the ACM}, vol.~12, no.~4, pp. 547--560, 1965.

\bibitem{abe2020linearly}
J.~Abe, M.~Yamagishi, and I.~Yamada, ``Linearly involved generalized {M}oreau enhanced models and their proximal splitting algorithm under overall convexity condition,'' \emph{Inverse Problems}, vol.~36, no.~3, p. 035012, 2020.

\bibitem{abe2019convexity}
------, ``Convexity-edge-preserving signal recovery with linearly involved generalized minimax concave penalty function,'' in \emph{ICASSP 2019-2019 IEEE International Conference on Acoustics, Speech and Signal Processing (ICASSP)}.\hskip 1em plus 0.5em minus 0.4em\relax IEEE, 2019, pp. 4918--4922.

\bibitem{yata2022constrained}
W.~Yata, M.~Yamagishi, and I.~Yamada, ``A constrained {LiGME} model and its proximal splitting algorithm under overall convexity condition.'' \emph{Journal of Applied \& Numerical Optimization}, vol.~4, no.~2, 2022.

\bibitem{al2021sharpening}
A.~H. Al-Shabili, Y.~Feng, and I.~Selesnick, ``Sharpening sparse regularizers via smoothing,'' \emph{IEEE Open Journal of Signal Processing}, vol.~2, pp. 396--409, 2021.

\bibitem{10251457}
Y.~Zhang and I.~Yamada, ``A unified framework for solving a general class of nonconvexly regularized convex models,'' \emph{IEEE Transactions on Signal Processing}, vol.~71, pp. 3518--3533, 2023.

\bibitem{chatterjee2020review}
S.~Chatterjee, R.~S. Thakur, R.~N. Yadav, L.~Gupta, and D.~K. Raghuvanshi, ``Review of noise removal techniques in {ECG} signals,'' \emph{IET Signal Processing}, vol.~14, no.~9, pp. 569--590, 2020.

\bibitem{bauschke2011convex}
H.~H. Bauschke and P.~L. Combettes, \emph{Convex analysis and monotone operator theory in {H}ilbert spaces}.\hskip 1em plus 0.5em minus 0.4em\relax Springer, 2011, vol. 408.

\bibitem{o2020equivalence}
D.~O’Connor and L.~Vandenberghe, ``On the equivalence of the primal-dual hybrid gradient method and {D}ouglas-{R}achford splitting,'' \emph{Mathematical Programming}, vol. 179, no. 1-2, pp. 85--108, 2020.

\bibitem{saad1986gmres}
Y.~Saad and M.~H. Schultz, ``{GMRES}: A generalized minimal residual algorithm for solving nonsymmetric linear systems,'' \emph{SIAM Journal on Scientific and Statistical Computing}, vol.~7, no.~3, pp. 856--869, 1986.

\bibitem{walker2011anderson}
H.~F. Walker and P.~Ni, ``Anderson acceleration for fixed-point iterations,'' \emph{SIAM Journal on Numerical Analysis}, vol.~49, no.~4, pp. 1715--1735, 2011.

\bibitem{potra2013characterization}
F.~A. Potra and H.~Engler, ``A characterization of the behavior of the {A}nderson acceleration on linear problems,'' \emph{Linear Algebra and its Applications}, vol. 438, no.~3, pp. 1002--1011, 2013.

\bibitem{fang2009two}
H.-r. Fang and Y.~Saad, ``Two classes of multisecant methods for nonlinear acceleration,'' \emph{Numerical Linear Algebra with Applications}, vol.~16, no.~3, pp. 197--221, 2009.

\bibitem{toth2015convergence}
A.~Toth and C.~T. Kelley, ``Convergence analysis for {A}nderson acceleration,'' \emph{SIAM Journal on Numerical Analysis}, vol.~53, no.~2, pp. 805--819, 2015.

\bibitem{chen2019convergence}
X.~Chen and C.~T. Kelley, ``Convergence of the {EDIIS} algorithm for nonlinear equations,'' \emph{SIAM Journal on Scientific Computing}, vol.~41, no.~1, pp. A365--A379, 2019.

\bibitem{de2022linear}
H.~De~Sterck and Y.~He, ``Linear asymptotic convergence of {A}nderson acceleration: fixed-point analysis,'' \emph{SIAM Journal on Matrix Analysis and Applications}, vol.~43, no.~4, pp. 1755--1783, 2022.

\bibitem{evans2020proof}
C.~Evans, S.~Pollock, L.~G. Rebholz, and M.~Xiao, ``A proof that {A}nderson acceleration improves the convergence rate in linearly converging fixed-point methods (but not in those converging quadratically),'' \emph{SIAM Journal on Numerical Analysis}, vol.~58, no.~1, pp. 788--810, 2020.

\bibitem{rebholz2023effect}
L.~G. Rebholz and M.~Xiao, ``The effect of {A}nderson acceleration on superlinear and sublinear convergence,'' \emph{Journal of Scientific Computing}, vol.~96, no.~2, p.~34, 2023.

\bibitem{scieur2016regularized}
D.~Scieur, A.~d'Aspremont, and F.~Bach, ``Regularized nonlinear acceleration,'' \emph{Advances In Neural Information Processing Systems}, vol.~29, 2016.

\bibitem{henderson2019damped}
N.~C. Henderson and R.~Varadhan, ``Damped {A}nderson acceleration with restarts and monotonicity control for accelerating {EM} and {EM}-like algorithms,'' \emph{Journal of Computational and Graphical Statistics}, vol.~28, no.~4, pp. 834--846, 2019.

\bibitem{mai2020anderson}
V.~Mai and M.~Johansson, ``Anderson acceleration of proximal gradient methods,'' in \emph{International Conference on Machine Learning}.\hskip 1em plus 0.5em minus 0.4em\relax PMLR, 2020, pp. 6620--6629.

\bibitem{bertrand2021anderson}
Q.~Bertrand and M.~Massias, ``Anderson acceleration of coordinate descent,'' in \emph{International Conference on Artificial Intelligence and Statistics}.\hskip 1em plus 0.5em minus 0.4em\relax PMLR, 2021, pp. 1288--1296.

\bibitem{zhang2019accelerating}
J.~Zhang, Y.~Peng, W.~Ouyang, and B.~Deng, ``Accelerating {ADMM} for efficient simulation and optimization,'' \emph{ACM Transactions on Graphics}, vol.~38, no.~6, pp. 1--21, 2019.

\bibitem{wang2021asymptotic}
D.~Wang, Y.~He, and H.~De~Sterck, ``On the asymptotic linear convergence speed of {A}nderson acceleration applied to {ADMM},'' \emph{Journal of Scientific Computing}, vol.~88, no.~2, p.~38, 2021.

\bibitem{he2022gda}
H.~He, S.~Zhao, Y.~Xi, C.~J. Ho, and Y.~Saad, ``Solve {M}inimax optimization by {A}nderson acceleration,'' in \emph{The Tenth International Conference on Learning Representations}, 2022.

\bibitem{wei2021stochastic}
F.~Wei, C.~Bao, and Y.~Liu, ``Stochastic {A}nderson mixing for nonconvex stochastic optimization,'' \emph{Advances in Neural Information Processing Systems}, vol.~34, pp. 22\,995--23\,008, 2021.

\bibitem{bollapragada2023nonlinear}
R.~Bollapragada, D.~Scieur, and A.~d’Aspremont, ``Nonlinear acceleration of momentum and primal-dual algorithms,'' \emph{Mathematical Programming}, vol. 198, no.~1, pp. 325--362, 2023.

\bibitem{zhang2020globally}
J.~Zhang, B.~O'Donoghue, and S.~Boyd, ``Globally convergent type-{I} {A}nderson acceleration for nonsmooth fixed-point iterations,'' \emph{SIAM Journal on Optimization}, vol.~30, no.~4, pp. 3170--3197, 2020.

\bibitem{fu2020anderson}
A.~Fu, J.~Zhang, and S.~Boyd, ``Anderson accelerated {D}ouglas-{R}achford splitting,'' \emph{SIAM Journal on Scientific Computing}, vol.~42, no.~6, pp. A3560--A3583, 2020.

\bibitem{zhou2014regularized}
H.~Zhou and L.~Li, ``Regularized matrix regression,'' \emph{Journal of the Royal Statistical Society Series B: Statistical Methodology}, vol.~76, no.~2, pp. 463--483, 2014.

\bibitem{simon2013sparse}
N.~Simon, J.~Friedman, T.~Hastie, and R.~Tibshirani, ``A sparse-group lasso,'' \emph{Journal of computational and graphical statistics}, vol.~22, no.~2, pp. 231--245, 2013.

\bibitem{klosa2020seagull}
J.~Klosa, N.~Simon, P.~O. Westermark, V.~Liebscher, and D.~Wittenburg, ``Seagull: lasso, group lasso and sparse-group lasso regularization for linear regression models via proximal gradient descent,'' \emph{BMC bioinformatics}, vol.~21, pp. 1--8, 2020.

\bibitem{fan2008sure}
J.~Fan and J.~Lv, ``Sure independence screening for ultrahigh dimensional feature space,'' \emph{Journal of the Royal Statistical Society Series B: Statistical Methodology}, vol.~70, no.~5, pp. 849--911, 2008.

\bibitem{malitsky2020forward}
Y.~Malitsky and M.~K. Tam, ``A forward-backward splitting method for monotone inclusions without cocoercivity,'' \emph{SIAM Journal on Optimization}, vol.~30, no.~2, pp. 1451--1472, 2020.

\bibitem{chen2023unified}
Y.~Chen, M.~Yamagishi, and I.~Yamada, ``A unified design of generalized {M}oreau enhancement matrix for sparsity aware {LiGME} models,'' \emph{IEICE Transactions on Fundamentals of Electronics, Communications and Computer Sciences}, p. 2022EAP1118, 2023.

\bibitem{combettes2001quasi}
P.~L. Combettes, ``Quasi-{F}ej{\'e}rian analysis of some optimization algorithms,'' in \emph{Studies in Computational Mathematics}.\hskip 1em plus 0.5em minus 0.4em\relax Elsevier, 2001, vol.~8, pp. 115--152.

\end{thebibliography}

\end{document}